\documentclass{amsart}

\usepackage{mathptmx}
\usepackage{helvet}
\usepackage{courier}
\usepackage{graphicx}
\usepackage{multicol}
\usepackage{footmisc}
\usepackage[T1]{fontenc}
\usepackage{euscript}
\usepackage{eufrak}
\usepackage{epic}

\usepackage{amsfonts,amsthm,amsmath}
\usepackage{amssymb}
\usepackage{amscd}
\usepackage{enumerate}
\usepackage[usenames,dvipsnames]{color}
\usepackage{tikz}
\usepackage{soul}
\usepackage{array}
\usepackage{array,tabularx}

\numberwithin{equation}{section}
\numberwithin{equation}{subsection}

\theoremstyle{plain}

\newtheorem{theorem}[equation]{Theorem}
\newtheorem{lemma}[equation]{Lemma}
\newtheorem{proposition}[equation]{Proposition}

\newtheorem{corollary}[equation]{Corollary}

\newtheorem{question}[equation]{Question}

\theoremstyle{definition}
\newtheorem{example}[equation]{Example}
\newtheorem{remark}[equation]{Remark}

\newtheorem{definition}[equation]{Definition}
\newtheorem{problem}[equation]{Problem}

\newtheorem{bekezdes}[equation]{}

\newcommand{\bZ}{{\mathbb Z}}

\newcommand{\calC}{{\mathcal C}}
\newcommand{\cJ}{{\mathcal J}}

\newcommand{\cV}{{\mathcal V}}

\newcommand{\cS}{{\mathcal S}}\newcommand{\calS}{{\mathcal S}}
\newcommand{\cP}{{\mathcal P}}
\newcommand{\cI}{{\mathcal I}}

\newcommand{\cO}{{\mathcal O}}\newcommand{\calO}{{\mathcal O}}
\newcommand{\cF}{{\mathcal F}}

\newcommand{\cL}{{\mathcal L}}

\newcommand{\calL}{{\mathcal L}}

\newcommand{\calQ}{{\mathcal Q}}

\newcommand{\tX}{\widetilde{X}}

\newcommand{\C}{{\calc}}

\newcommand{\rank}{{\rm rank}\, }

\newcommand{\ix}{\index}

\def\blfootnote{\xdef\@thefnmark{}\@footnotetext}

\newcommand{\co}{\cO}

\newcommand{\chic}{\mathfrak{r}}

\newcommand{\bt}{{\bf t}}

\newcommand{\bH}{{\mathbb H}}
\newcommand{\ocalj}{\overline{{\mathcal V}}}
\newcommand{\cali}{{\mathcal I}}
\newcommand{\calj}{{\mathcal V}}

\newcommand{\setQ}{\mathbb{Q}}

\newcommand{\setZ}{\mathbb{Z}}

\newcommand{\ringd}{\setZ[[\bt^{\pm 1/d}]]}
\newcommand{\hh}{\mathfrak{h}}
\newcommand{\pp}{\mathfrak{p}}
\newcommand{\RR}{\mathfrak{R}}

\newcommand{\calv}{\mathcal{V}}

\newcommand{\sw}{{\mathfrak{sw}}}

\newcommand{\Z}{\mathbb{Z}}
\newcommand{\Q}{\mathbb{Q}}
\newcommand{\R}{\mathbb{R}}

\begingroup\makeatletter\ifx\SetFigFontNFSS\undefined%
\gdef\SetFigFontNFSS#1#2#3#4#5{%
  \reset@font\fontsize{#1}{#2pt}%
  \fontfamily{#3}\fontseries{#4}\fontshape{#5}%
  \selectfont}%
\fi\endgroup%

\newcommand{\calt}{{\mathcal T}}

\def\C{\mathbb C}
\def\Q{\mathbb Q}
\def\R{\mathbb R}
\def\bH{\mathbb H}

\def\Z{\mathbb Z}

\newcommand{\cale}{{\mathcal E}}

\author{Tam\'as \'Agoston}
\address{Alfr\'ed R\'enyi Institute of Mathematics,
Re\'altanoda utca 13-15, H-1053, Budapest, Hungary }
\email{agoston.tamas@renyi.hu}

\author{Andr\'as N\'emethi}
\address{Alfr\'ed R\'enyi Institute of Mathematics,
Re\'altanoda utca 13-15, H-1053, Budapest, Hungary \newline
 \hspace*{4mm} ELTE - University of Budapest, Dept. of Geometry, Budapest, Hungary \newline \hspace*{4mm}
BCAM - Basque Center for Applied Math.,
Mazarredo, 14 E48009 Bilbao, Basque Country – Spain}
\email{nemethi.andras@renyi.hu }

\title{Analytic lattice cohomology of surface singularities, II\\
(the equivariant case)
}

\begin{document}

\thanks{The  authors are partially supported by NKFIH Grant ``\'Elvonal (Frontier)'' KKP 126683.}

\begin{abstract}
We construct the equivariant analytic lattice cohomology associated with the analytic type of a complex normal surface singularity
whenever the link is a rational homology sphere. It is the categorification of the equivariant geometric genus of the  germ.
This is the analytic analogue of the topological lattice cohomology, associated with the link of the germ,
and  indexed by the spin$^c$--structures of the link (which is a categorification of the Seiberg--Witten invariant and
conjecturally it is isomorphic with the Heegaard Floer cohomology).

\end{abstract}

\maketitle


\date{}

\pagestyle{myheadings} \markboth{{\normalsize  T. \'Agoston , A. N\'emethi}} {{\normalsize Analytic lattice cohomology }}


\section{Introduction}\label{s:intr}

\subsection{} Let us fix a complex normal surface singularity $(X,o)$ whose link is a rational homology
sphere. In \cite{NOSz,NGr,Nlattice} the (topological) lattice cohomologies and graded roots were introduced (using the combinatorics of the dual graph of any good resolution).
 Let us recall some of its main properties.

 It  has a  rather different structure than any cohomology
theory associated with  analytic spaces by complex analytic or algebraic geometry.
It  has several  gradings: first of all, it has a direct sum decomposition according to the spin$^c$--structures $\sigma$
of $M$. (Recall that ${\rm Spin}^c(M)$  is an $H_1(M,\Z)$ torsor, hence the cardinality
of ${\rm Spin}^c(M)$ is the order of $H_1(M,\Z)$.)
 Then each summand $\bH^*_{top}(M,\sigma)$ has a decomposition
$\oplus _{q\geq 0}\bH^q(M,\sigma)$, where each $\bH^q_{top}(M,\sigma)$ is a $\Z$--graded $\Z[U]$--module.
Probably the presence of this  additional $U$--action is the most outstanding property compared with the usual cohomology theories.

Conjecturally (see \cite{Nlattice})
$\bH^*_{top}(M)$  is isomorphic to the Heegaard Floer cohomology $HF^+$ of Ozsv\'ath and Szab\'o  (which is defined for any 3--manifold),
for  $HF$--theory see their
 long list of article, e.g.
\cite{OSz,OSz7}. This conjecture was verified for several families of plumbed 3--manifolds (associated with
negative definite connected graphs),  cf. \cite{NOSz,OSSz3},
but the general case is still open.
(In fact,  the Heegaard Floer theory is isomorphic with several other theories:
with the Monopole Floer Homology of Kronheimer  and Mrowka, or with the Embedded Contact Homology of Hutchings.
They are based on different geometrical  aspects of the 3--manifold $M$.)
$\bH^*_{top}$  is the categorification of the Seiberg--Witten invariant (similarly as $HF^+$ is).
(This means that the Euler characteristic of $\bH^*_{top}$ is the Seiberg--Witten invariant.)
For several properties and application in singularity theory see \cite{NOSz,NSurgd,NGr,Nexseq,NeLO}. For its connection with the classification
projective rational plane cuspidal curves (via superisolated surface  singularities) see \cite{NSurgd,BLMN2,BodNem,BodNem2,BCG,BL1}.
 It  provides sharp topological bounds for certain sheaf cohomologies (e.g. for $p_g$), see e.g.  \cite{NSig,NSigNN}.
An improvement of $\bH^0_{top}$ is the set of graded roots parametrized by the spin$^c$--structures of $M$ \cite{NOSz,NGr}
 (they have no analogues for general arbitrary 3--manifolds). The  graded root is a special
tree with  $\Z$--graded vertices, it  provides
a very visual presentation of  $\bH^0_{top}$ (e.g., the $U$--action is coded in the edges).
 Hence, in particular it visualizes   $HF^+$ too,
when the  Heegaard Floer homology is known to be isomorphic to $\bH^0_{top}$ (see e.g. \cite{NOSz}). In such cases
the use of   graded roots  is significantly more convenient than any other method, see e.g.
 \cite{DM,K1,K2,K3}.

 \subsection{} In a series of articles we wish to develop the theory of analytic lattice
 cohomologies: they are associated with the analytic type of isolated singularities of
 any dimension, see
 \cite{AgNe1,AgNeCurves,AgNeHigh}.

 In \cite{AgNe1} we considered the case of a normal surface singularities,  when we constructed the analytic lattice cohomology associated with the
 canonical spin$^c$--structure.
  The case of  other spin$^c$--structures (under the assumption that the link is a rational homology sphere)
  is  treated in the present note.
For this general part,
 we need to generalize  the constructions of \cite{AgNe1}
 to the level  of the universal abelian covering of $(X,o)$ and we also
need to use  several technical
parts regarding  `natural line bundles'  of a
resolution. This motivates that this equivariant discussion is separated in the present note.

The analytic lattice cohomology $\bH^*_{an}(X,o)$ has a very similar structure as the topological one. It
decomposes into a direct sum, where the summands are  indexed
by the elements of $H_1(M,\Z)$ (hence, equivalently, by ${\rm Spin}^c(M)$),
and each summand is a double graded $\Z[U]$--module.
The cohomology theory is the categorification of equivariant geometric genus.
 We also show that it admits a graded $\Z[U]$--module morphism $\bH^*_{an}(X,o)\to \bH^*_{top}(M)$.
 We also present a reduction theorem similar to the non-equivariant case (and
 comparable with the topological case \cite{LN1}).


 \subsection{} The structure of the article is the following.

 In section 2 we recall the general definition of lattice cohomology (and graded root)
 associated with a weight  function. For this construction we need a free module $\Z^s$ (with fixed basis)
 and a weight function $w:\Z^s\to\Z$. In both topological and analytical cases the lattice $\Z^s$ is given  by $H_2(\tX,\Z)$ of a good resolution $\tX\to X$. However, in the topological case, the weight function is determined topologically, and in the analytic case it is analytic: it is the difference  of the coefficient of the Hilbert function and the dimension of a  sheaf cohomology.

 In section 3 we prove  combinatorial theorems regarding the Euler characteristic
 of a lattice cohomology associated with a weight function with certain `nice' properties.

 In section 4 we review properties of the topological lattice cohomology.

 In section 5 we collected certain needed terminologies, analytic results and constructions
 (universal abelian covering, equivariant geometric genus, natural line bundles, equivariant
 multivariable Hilbert series, vanishing and duality theorems, and
cohomological cycle associated with a line bundle).

Section 6 contains the definition of the analytic lattice cohomology using a resolution. Here we also prove its  independence  of the choice of the resolution and we determine its Euler characteristic.

In section 7 we construct a graded $\Z[U]$--module morphism $\mathfrak{H}^*_h:\bH^*  _{an,h}(X,o)\to
\bH^*_{top,h}(M)$.

In section 8 we review
 the topological reduction theorem (reduction to a smaller rank  lattice associated with the set of
`bad' vertices). Section 9 contains the analytic version of this.

\section{Preliminaries. Basic properties of lattice cohomology}\label{s:Prem1}

This is a short review of the lattice cohomology and graded roots associated with a weight function.
Though this material was presented in many different articles, still is worth to
recall the notations and basic results in order to make the next sections readable.
This section is rather similar  with section 2 of \cite{AgNe1}.

\subsection{The lattice cohomology associated with  a weight function}\label{ss:latweight} \cite{NOSz,Nlattice}

\bekezdes {\bf Weight function.}
 We consider a free $\Z$-module, with a fixed basis
$\{E_v\}_{v\in\calv}$, denoted by $\Z^s$, $s:=|\calv|$.
Additionally, we consider a {\it weigh function} $w_0:\Z^s\to \Z$ with the property
\begin{equation}\label{9weight}
\mbox{for any integer $n\in\Z$, the set $w_0^{-1}(\,(-\infty,n]\,)$
is finite.}\end{equation}

%

\bekezdes\label{9complex} {\bf The weighted cubes.}
The space
$\Z^s\otimes \R$ has a natural cellular decomposition into cubes. The
set of zero-dimensional cubes is provided  by the lattice points
$\Z^s$. Any $l\in \Z^s$ and subset $I\subset \calv$ of
cardinality $q$  defines a $q$-dimensional cube $\square_q=(l, I)$, which has its
vertices in the lattice points $(l+\sum_{v\in I'}E_v)_{I'}$, where
$I'$ runs over all subsets of $I$.
 The set of $q$-dimensional cubes  is denoted by $\calQ_q$ ($0\leq q\leq s$).

Using $w_0$ we define
$w_q:\calQ_q\to \Z$  ($0\leq q\leq s$) by
$w_q(\square_q):=\max\{w_0(l)\,:\, \mbox{$l$ is a vertex of $\square_q$}\}$.

For each $n\in \Z$ we
define $S_n=S_n(w)\subset \R^s$ as the union of all
the cubes $\square_q$ (of any dimension) with $w(\square_q)\leq
n$. Clearly, $S_n=\emptyset$, whenever $n<m_w:=\min\{w_0\}$. For any  $q\geq 0$, set
$$\bH^q(\R^s,w):=\oplus_{n\geq m_w}\, H^q(S_n,\Z)\ \ \mbox{and}\ \
\bH^q_{red}(\R^s,w):=\oplus_{n\geq m_w}\, \widetilde{H}^q(S_n,\Z).$$
Then $\bH^q$ is $\Z$ (in fact, $2\Z$)-graded, the
$2n$-homogeneous elements $\bH^q_{2n}$ consist of  $H^q(S_n,\Z)$.
Also, $\bH^q$ is a $\Z[U]$-module; the $U$-action is given by
the restriction map $r_{n+1}:H^q(S_{n+1},\Z)\to H^q(S_n,\Z)$.
Namely,  $U*(\alpha_n)_n=(r_{n+1}\alpha_{n+1})_n$. The same is true for $\bH^*_{red}$.
 Moreover, for
$q=0$, the fixed base-point $l_w\in S_n$ provides an augmentation
(splitting)
 $H^0(S_n,\Z)=
\Z\oplus \widetilde{H}^0(S_n,\Z)$, hence an augmentation of the graded
$\Z[U]$-modules (where
$\calt_{2m}^+= \Z\langle U^{-m}, U^{-m-1},\ldots\rangle$ as a $\Z$-module with its natural $U$--action)
$$\bH^0\simeq\calt^+_{2m_w}\oplus \bH^0_{red}=(\oplus_{n\geq m_w}\Z)\oplus (
\oplus_{n\geq m_w}\widetilde{H}^0(S_n,\Z))\ \ \mbox{and} \ \
\bH^*\simeq \calt^+_{2m_w}\oplus \bH^*_{red}.$$

Though
$\bH^*_{red}(\R^s,w)$ has finite $\Z$-rank in any fixed
homogeneous degree, in general,  without certain additional properties of $w_0$, it is not
finitely generated over $\Z$, in fact, not even over $\Z[U]$.

\bekezdes\label{9SSP} {\bf Restrictions.} Assume that $T\subset \R^s$ is a subspace
of $\R^s$ consisting of a union of some cubes (from $\calQ_*$). For any $q\geq 0$ define $\bH^q(T,w)$ as
$\oplus_{n\geq\min{w_0|T}} H^q(S_n\cap T,\Z)$. It has a natural graded $\Z[U]$-module
structure.  The restriction map induces a natural graded
$\Z[U]$-module homogeneous homomorphism
$$r^*:\bH^*(\R^s,w)\to \bH^*(T,w) \ \ \ \mbox{(of degree zero)}.$$
In our applications to follow, $T$ (besides the trivial $T=\R^s$ case) will be one of the following:
%
(i)  the first quadrant $(\R_{\geq o})^s$,
(ii) the rectangle $[0,c]=\{x\in \R^s\,:\, 0\leq x\leq c\}$ for some lattice point $c\geq 0$, or
(iii)  a path of composed edges in the lattice, cf. \ref{9PCl}.

\bekezdes \label{9F} {\bf The `Euler characteristic' of $\bH^*$.}
Fix $T$ as in  \ref{9SSP} and we will assume that each $\bH^*_{red}(T,w)$ has finite $\Z$--rank.
%
The Euler characteristic of $\bH^*(T,w)$ is defined as
$$eu(\bH^*(T,w)):=-\min\{w(l)\,:\, l\in T\cap \Z^s\} +
\sum_q(-1)^q\rank_\Z(\bH^q_{red}(T,w)).$$

\begin{lemma}\cite{NJEMS}\label{bek:LCSW} If $T=[0,c]$ for a lattice point $c\geq 0$, then
\begin{equation}\label{eq:Ecal}
 \sum_{\square_q\subset T} (-1)^{q+1}w_k(\square_q)=eu(\bH^*(T,w)).\end{equation}
 \end{lemma}

\subsection{Path lattice cohomology}\label{9PCl}\cite{Nlattice}

\bekezdes \label{bek:pathlatticecoh}
Fix $\Z^s$  as in \ref{ss:latweight} and fix also a compatible weight functions
 $\{w_q\}_q$   as in \ref{9weight}. 
 Consider also a sequence $\gamma:=\{x_i\}_{i=0}^t$  so that $x_0=0$,
$x_i\not=x_j$ for $i\not=j$, and $x_{i+1}=x_i\pm E_{v(i)}$ for
$0\leq i<t$. We write $T$ for the
union of 0-cubes marked by the points $\{x_i\}_i$ and of the
segments of type  $[x_i,x_{i+1}]$.
Then, by \ref{9SSP} we get a graded $\Z[U]$-module $\bH^*(T,w)$,
which is called the {\em
path lattice cohomology} associated with the `path' $\gamma$ and weights
$\{w_q\}_{q=0,1}$. It is denoted by $\bH^*(\gamma,w)$.
It has an augmentation with $\calt^+_{2m_\gamma}$,
where $m_\gamma:=\min_i\{w_0(x_i)\}$, and one gets the {\em reduced path lattice
cohomology} $\bH^0_{red}(\gamma,w)$ with
$$\bH^0(\gamma,w)\simeq\calt_{2m_\gamma}^+\oplus
\bH^0_{red}(\gamma,w).$$
It turns out that  $\bH^q(\gamma,w)=0$ for $q\geq 1$, hence its `Euler characteristic' can be defined as  (cf.  \ref{9F})
\begin{equation}\label{eq:euh0}
eu(\bH^*(\gamma,w)):=-m_\gamma+\rank_\Z\,(\bH^0_{red}(\gamma,w)).\end{equation}

\begin{lemma} \label{9PC2}
One has the following expression of $eu(\bH^*(\gamma,w))$ in terms of the values of $w$:
\begin{equation}\label{eq:pathweights}
eu(\bH^*(\gamma,w))=-w_0(0)+\sum_{i=0}^{t-1}\,
\max\{0, w_0(x_{i})-w_0(x_{i+1})\}.
\end{equation}
\end{lemma}


\subsection{Graded roots and their cohomologies}\label{s:grgen} \cite{NOSz,NGr}

\begin{definition}\label{def:2.1} \
  Let $\RR$ be an infinite tree with vertices $\calv$ and edges
$\cale$. We denote by $[u,v]$ the edge with
 end-vertices  $u$ and $v$.  We say that $\RR$ is a {\em graded root}
with grading $\chic:\calv\to \Z$ if

(a) $\chic(u)-\chic(v)=\pm 1$ for any $[u,v]\in \cale$;

(b) $\chic(u)>\min\{\chic(v),\chic(w)\}$ for any $[u,v],\
[u,w]\in\cale$, $v\neq w$;

(c) $\chic$ is bounded below, $\chic^{-1}(n)$ is finite for any
$n\in\Z$, and $|\chic^{-1}(n)|=1$ if $n\gg 0$.

%

\vspace{1mm}

\noindent
An isomorphism of graded roots is a graph isomorphism, which preserves the gradings.
\end{definition}

\begin{definition}\label{9.2.6}{\bf The
  $\Z[U]$-modules associated with a graded root.}
Let us identify a graded root $(\RR,\chic)$ with its topological realization provided by vertices (0--cubes)
and segments (1--cubes). Define $w_0(v)=\chic(v)$, and $w_1([u,v])=\max\{\chic(u),\chic(v)\}$ and let $S_n$ be the
union of all cubes with weight $\leq n$. Then  we might set (as above) $\bH^*(\RR,\chi)=\oplus_{n\geq \min\chic}\
H^*(S_n,\Z)$. However, at this time $\bH^{\geq 1}(\RR,\chic)=0$; we set $\bH(\RR,\chic):=\bH^0(\RR,\chic)$.
Similarly, one defines $\bH_{red}(\RR,\chic)$ using the reduced cohomology, hence
$\bH(\RR,\chic)\simeq\calt_{2\min \chic}^+\oplus \bH_{red}(\RR,\chic)$.
\end{definition}

For a detailed concrete description of $\bH(\RR)$ in terms of the combinatorics of the root see \cite{NOSz}.

\bekezdes\label{bek:GRootW}{\bf The graded root associated with a weight function.}
Fix a free $\Z$-module and a system of weights $\{w_q\}_q$.
Consider the sequence of  topological  spaces (finite cubical  complexes) $\{S_n\}_{n\geq m_w}$
with $S_n\subset S_{n+1}$, cf. \ref{9complex}.
Let $\pi_0(S_n)=\{\calC_n^1,\ldots , \calC_n^{p_n}\}$ be the set of connected components of $S_n$.

Then  we define the  graded graph  $(\RR_w,\chic_w)$ as follows. The
vertex set  $\calv(\RR_w)$ is $\cup_{n\in \Z} \pi_0(S_n)$.
The grading $\chic_w:\calv(\RR_w)\to\Z$ is
$\chic_w(\calC_n^j)=n$, that is, $\chic_w|_{\pi_0(S_n)}=n$.
Furthermore, if  $\calC_{n}^i\subset \calC_{n+1}^j$ for some $n$, $i$ and $j$,
then we introduce an edge $[\calC_n^i,\calC_{n+1}^j]$. All the edges
 of $\RR_w$ are obtained in this way.
\begin{lemma}\label{lem:GRoot} $(\RR_w,\chic_w)$ satisfies all the required properties
of the definition of a graded root, except maybe the last one: $|\chic_w^{-1}(n)|=1$ whenever $n\gg 0$.
\end{lemma}

The property  $|\chic_w^{-1}(n)|=1$ for $n\gg 0$ is not always satisfied.
However, the graded roots associated with connected negative definite plumbing graphs
(see below) satisfies this condition as well.

\begin{proposition}\label{th:HHzero} If $\RR$ is a graded root associated with $(T,w)$ and
 $|\chic_w^{-1}(n)|=1$ for all $n\gg 0$ then $\bH(\RR)=\bH^0(T,w)$.
\end{proposition}

\section{Combinatorial lattice cohomology} \label{ss:CombLattice}

\subsection{}
In this section we review several combinatorial statements regarding the lattice cohomology
associated with any weight function with certain combinatorial  properties. We follow \cite{AgNe1}.


\bekezdes \label{bek:comblattice}
Fix  $\Z^s$ with a fixed basis $\{E_v\}_{v\in\cV}$.
Write $E_I=\sum_{v\in I}E_v$ for $I\subset \cV$ and  $E=E_{\cV}$.
Fix also  an element
$c\in \Z^s$, $c\geq E$.
Consider the lattice points $R=R(0,c):=\{l\in\Z^s\,:\, 0\leq l\leq c\}$, and assume that
to each $l\in R$ we assign

(i)   an integer $h(l)$ such that $h(0)=0$ and $h(l+E_v)\geq h(l)$
for any $v$,

(ii)  an integer $h^\circ (l)$ such that $h^\circ (l+E_v)\leq  h^\circ (l)$
for any $v$.


Once  $h$ is fixed with (i),
a possible choice for $h^\circ $ is
 $h^{sym}$, where $h^{sym}(l)=h(c-l)$. Clearly, it depends on $c$.

\bekezdes
 We say that the $h$-function 
 satisfies the  {\it `matroid inequality'} if
 \begin{equation}\label{eq:matroid}
 h(l_1)+h(l_2)\geq h(\min\{l_1,l_2\})+h(\max\{l_1,l_2\}), \ \ l_1,l_2\in R.
 \end{equation}
 This implies the {\it `stability property'},
 valid for any $\bar{l}\geq 0$ with $|\bar{l}|\not\ni E_v$
 \begin{equation}\label{eq:stability}
 h(l)=h(l+E_v)\ \ \Rightarrow\  \ h(l+\bar{l})=h(l+\bar{l}+E_v).
 \end{equation}
If $\hh$ is given by a filtration (see below) then it automatically satisfies the matroid inequality.

\bekezdes
We  consider the set of cubes $\{\calQ_q\}_{q\geq 0}$ of $R$ as in \ref{9complex} and
the weight function
$$w_0:\calQ_0\to\Z\ \ \mbox{by} \ \ w_0(l):=h(l)+h^\circ (l)-h^\circ (0).$$
Clearly  $w_0(0)=0$.
Furthermore, we define
$w_q:\calQ_q\to \Z$ by $ w_q(\square_q)=\max\{w_0(l)\,:\, l \
 \mbox{\,is a vertex of $\square_q$}\}$. We will use the symbol $w$ for the
 system $\{w_q\}_q$.
 The compatible weight functions define the lattice cohomology $\bH^*(R,w)$.
Moreover, for any increasing path $\gamma$ connecting 0 and $c$
 we also have
 a path lattice cohomology $\bH^0(\gamma,w)$ as in \ref{bek:pathlatticecoh}. Accordingly,  we have the numerical
Euler characteristics  $eu(\bH^*(R,w))$, $eu(\bH^0(\gamma,w))$ and $\min_\gamma eu(\bH^0(\gamma,w))$ too.

\begin{lemma}\label{lem:comblat} \ \cite{AgNe1} We have   $0\leq eu(\bH^0(\gamma,w))\leq h^\circ (0)-h^\circ (c)$
for any increasing path $\gamma$  connecting  0 to $c$.
 The equality $eu(\bH^0(\gamma,w))=h^\circ (0)-h^\circ (c)$
holds if and only if for any $i$
the differences $h(x_{i+1})-h(x_i)$ and $h^\circ (x_{i})-h^\circ (x_{i+1})$ simultaneously are not nonzero.
\end{lemma}

\begin{definition}\label{def:COMPGOR}
Fix  $(h,h^\circ,R)$ as in \ref{bek:comblattice}.
 We say that the pair $h$ and $h^\circ$ satisfy the `Combinatorial Duality  Property' (CDP) if
$h(l+E_v)-h(l)$ and $h^\circ (l+E_v)-h^\circ (l)$ simultaneously cannot be nonzero
for $l,\, l+E_v\in R$. Furthermore,
 we say that $h$  satisfies  the CDP  if
 the pair $(h,h^{sym})$ satisfies  it.
\end{definition}

\begin{definition}\label{def:comblat}
We say that the pair   $(h, h^\circ) $ satisfy the

(a) {\it `path eu-coincidence'} if $eu(\bH^0(\gamma,w))=h^\circ (0)-h^\circ (c)$
for any increasing path $\gamma$.

(b)  {\it `eu-coincidence'} if $eu(\bH^*(R,w))=h^\circ (0)-h^\circ (c)$.
\end{definition}

\begin{remark}
 Example 4.3.3 of \cite{AgNe1}
 shows the following two facts.

Even if $h$ satisfies the path eu-coincidence (and $h^\circ =h^{sym}$),
in general it is not true that $\bH^0(\gamma,w)$
is independent of the choice of the increasing path.
(This statement remains valid even if we consider only the symmetric increasing paths, where a
 path $\gamma=\{x_i\}_{i=0}^t$ is symmetric if $x_{t-l}=c-x_l$ for any $l$.)

Even if $h$ satisfies both the path eu-coincidence and the eu-coincidence,
in general it is not true that $\bH^*(R,w)$ equals  any of the path lattice cohomologies
$\bH^0(\gamma,w)$ associated with a certain  increasing  path.
(E.g., in the mentioned Example 4.3.3  we have $\bH^1(R,w)\not=0$, a fact which does not hold for any
path lattice cohomology.) However, amazingly, all the Euler characteristics agree.
\end{remark}


\begin{theorem}\label{th:comblattice}
Assume that $h$ satisfies the stability property, and the pair $(h,h^\circ)$
satisfies the Combinatorial Duality  Property. Then the following facts hold.

\noindent (a) \  $(h,h^\circ)$
satifies both the path eu- and the eu-coincidence properties:  for any increasing $\gamma$  we have
$$eu(\bH^*(\gamma,w))=eu(\bH^*(R,w))=h^\circ (0)-h^\circ (c).$$
(b)
$$\sum_{l\geq 0}\,\sum _I\, (-1)^{|I|+1} w((l,I))\, \bt^{l}=
\sum_{l\geq 0}\,\sum _I\, (-1)^{|I|+1} h(l+E_I)\,\bt^{l}.$$
\end{theorem}

\section{Surface singularities and the topological lattice cohomology}\label{s:prel}

\subsection{The combinatorics  of a resolution}\cite{Nfive,NOSz,NGr}
\bekezdes
Let $(X,o)$ be the germ of a complex analytic normal surface singularity with link $M$.
Let $\phi:\widetilde{X}\to X$ be a good   resolution   of $(X,o)$ with
 exceptional curve $E:=\phi^{-1}(0)$,  and  let $\cup_{v\in\calv}E_v$ be
the irreducible decomposition of $E$. 
 Let $\Gamma$ be the dual resolution graph of $\phi$.  Note that $\partial \tX\simeq M$.

The lattice $L:=H_2(\widetilde{X},\mathbb{Z})$ is  endowed
with the natural  negative definite intersection form  $(\,,\,)$. It is
freely generated by the classes of  $\{E_v\}_{v\in\mathcal{V}}$.
 The dual lattice is $L'={\rm Hom}_\Z(L,\Z) \simeq\{
l'\in L\otimes \Q\,:\, (l',L)\in\Z\}$.
It  is generated
by the (anti)dual classes $\{E^*_v\}_{v\in\mathcal{V}}$ defined
by $(E^{*}_{v},E_{w})=-\delta_{vw}$ (where $\delta_{vw}$ stays for the  Kronecker symbol).
$L'$ is also  identified with $H^2(\tX,\Z)$. 

We define the Lipman cone as $\calS':=\{l'\in L'\,:\, (l', E_v)\leq 0 \ \mbox{for all $v$}\}$, and we  also
set $\calS:=\calS'\cap L$. If $s'\in\calS'\setminus \{0\}$ then
all its $E_v$--coordinates  are strict positive.

The intersection form embeds $L$ into $L'$ with
 $ L'/L\simeq {\rm Tors}( H_1(M,\mathbb{Z}))$, which is abridged by $H$.
 The class of $l'$ in $H$ is denoted by $[l']$.

There is a natural partial ordering of $L'$ and $L$: we write $l_1'\geq l_2'$ if
$l_1'-l_2'=\sum _v r_vE_v$ with every  $r_v\geq 0$. We set $L_{\geq 0}=\{l\in L\,:\, l\geq 0\}$ and
$L_{>0}=L_{\geq 0}\setminus \{0\}$.
The support of a cycle $l=\sum n_vE_v$ is defined as  $|l|=\cup_{n_v\not=0}E_v$.


The {\it (anti)canonical cycle} $Z_K\in L'$ is defined by the
{\it adjunction formulae}
$(Z_K, E_v)=(E_v,E_v)+2-2g_v$ for all $v\in\mathcal{V}$, where
$g_v$  denotes the genus of $E_v$.
The cycle $-Z_K$  is the first Chern class of the line bundle $\Omega^2_{\tX}$.
We write $\chi:L'\to \Q$ for the (Riemann--Roch) expression $\chi(l'):= -(l', l'-Z_K)/2$.


If  $H_1(M,\Q)=0$ then
each $E_v$ is rational, and the dual graph of any good resolution is a tree. In this case $H_1(M,\Z)=H$ is finite. In this case
we denote the Pontrjagin dual $\mathrm{Hom}(H,S^1)$ of $H$ by $\widehat{H}$.
Let $\theta:H\to \widehat{H}$ be the isomorphism $[l']\mapsto
e^{2\pi i(l',\cdot)}$ of $H$ with  $\widehat{H}$.

\begin{definition}\label{def:char}
The set of characteristic elements are defined as
\begin{equation}\label{eq:char}
{\rm Char}={\rm Char}(L)=\{k\in L'\, :\, (l,l+k)\in 2\bZ \ \ \mbox{for any $l\in L$}\}.
\end{equation}
\end{definition}
Note that $-Z_K\in {\rm Char}$   and ${\rm Char}=-Z_K+2L'$ (and
${\rm Char}$ is an $L'$ torsor by the  action $l'*k=k+2l'$).
The RR--expression $\chi$  has an analogue for any
$k\in {\rm Char}$, namely
one defines $\chi_k:L\to \Z$ \, by \, $\chi_k(l):=-(l,l+k)/2$.

\bekezdes\label{ss:mincyc} {\bf Canonical representatives and spin$^c$-structures.}
For any $h\in H$ there exists a unique element $r_h=\sum_vr_vE_v\in L'$ with $[r_h]=h$ such that
each $r_v\in[0,1)$.
Similarly,
 for any $h\in H$ there is a
 unique minimal element of $\{l'\in L' \ | \ [l']=h\}\cap \mathcal{S}'$.
 It  will be denoted by $s_h$. For $h=0$ we have $r_h=s_h=0$.
One has  $s_h\geq r_h$;  in general, $s_h \neq r_h$.

Assume that the link  is a rational homology sphere.  Then ${\rm Spin}^c(\tX)$, the set of spin$^c$--structures on $\tX$, is
identified with the
set of characteristic elements on $L'$. Moreover, any spin$^c$--structure on $\partial \tX=M$ is
the restriction of a spin$^c$--structure of $\tX$  and if  $k$ and $k'$ induces  the same
spin$^c$--structure on the link
then $k'=k+2l$ for a certain $l\in L$.
This is an equivalence relation on ${\rm Char}$, the classes are denoted by $[k]$.
 If $k'=k+2l$ for some $l\in L$ then
  $\chi_{k'}(x-l)=\chi_k(x)-\chi_k(l)$
for any $x\in L$, hence the two
functions $\chi_k$ and $\chi_{k'}$ can be easily compared, and they have identical
qualitative properties.
Therefore, for each class $[k]=k+2L$ (that is, for each spin$^c$--structure
$\sigma[k]$  of
$M$), we might choose a
representative of $[k]$. Since the set of classes is indexed by $H$; we define the set
of representatives
 by $k_r:=-Z_K+2s_h$, for each $h\in H$. Since $s_0=0$, for the trivial class
$h=0$ we get $\chi_{k_r}=\chi$. (This choice will produce several pleasant consequences, e.g. \ref{9STR2}{\it (d)}.)

\subsection{The topological lattice cohomology associated with $\phi:\tX\to X$}\label{s:latticeplgraphs}\cite{NOSz,Nlattice}

\bekezdes\label{9dEF1} We consider a good resolution $\phi$ as above
and we assume that the link  $M$ is a
rational homology sphere. We write $s:=|\cV|$.
We also fix a characteristic element $k\in {\rm Char}$.

Then we automatically have  a free $\Z$-module $L=\Z^s$ with
a fixed bases $\{E_v\}_v$, and  $k$ defines a set
of compatible weight functions $w$ by
$ w_k(\square_q)=\max\{\chi_k(v)\,:\, v\ \mbox{is a vertex of
$\square_q$}\}$.

\begin{definition}\label{9DEF}  The  $\Z[U]$-modules $\bH^*(\R^s,w)$ and
$\bH^*_{red}(\R^s,w)$ obtained by these weight functions are
called the {\em lattice cohomologies} associated with the pair
$(\phi,k)$ and are denoted by $\bH^*(\Gamma,k)$, respectively
$\bH^*_{red}(\Gamma,k)$.

The graded root associated with $(\Z^s,w_k)$ will be denoted by $\RR(\Gamma, k)$.
\end{definition}

\begin{proposition}\label{9STR2} \cite{NOSz,NGr,LN1}

(a) $\bH^*_{red}(\Gamma,k)$ is finitely generated over $\Z$.


(b) The set  $\bH^*(\Gamma,k_r)$ (indexed by the spin$^c$--structures of $M$)
depends only on $M$ and is independent of the choice of the good resolution $\phi$.
They are called the topological lattice cohomologies of the singularity $(X,o)$, or of the link $M$.
In the sequel we might also refer to is as $\bH^*(M,k_r)$.

(c) The restriction
$\bH^*(\Gamma,k_r)\to \bH^*((\R_{\geq 0})^s,k_r)$
induced by the inclusion $(\R_{\geq 0})^s\hookrightarrow \R^s$ is an isomorphism
of graded $\Z[U]$ modules.

There are similar statements for $\RR(\Gamma , k_r)$ instead of $\bH^*(\Gamma, k_r)$, which will
also  be denoted
by $\RR(M,k_r)$.
\end{proposition}

\bekezdes {\bf The Euler characteristic and the Seiberg--Witten invariant.}\label{s:LCSW} \
The Seiberg--Witten invariant ${\rm Spin}^c(M)\to \Q$ associates a rational number
$\sw _{\sigma}(M)$  to each spin$^c$--structure $\sigma$ of the link. Recall also
that  ${\rm Spin}^c(M)$ is an $H$--torsor, and it can be parametrized by the classes $[k]\in{\rm Char}/2L$, or by the representatives $\{k_r\}$.

\begin{theorem}\label{th:ECharLC} \cite{NJEMS} 
Let  $\sigma[k_r]$ be the spin$^c$--structure associated with $k_r$. Then
$$eu(\bH^*(M,k_r))=\sw_{\sigma[k_r]}(M)-\frac{k_r^2+|\cV|}{8}.$$
\end{theorem}
In other words, {\it the topological lattice cohomology is the categorification of the Seiberg--Witten
invariant   (normalized by $(k_r^2+|\cV|)/8$)}.

\begin{remark}\label{bek:rhred}
Consider the topological lattice cohomologies associated with characteristic elements
$-Z_K+2r_h$ and $-Z_K+2s_h$, and with cubes from $\R^s$ and $\R_{\geq 0}^s$.
 We claim that there also exists  a graded $\Z[U]$--module isomorphism (the analogue of
 Proposition \ref{9STR2}{\it (c)}):
 $$\bH^*(\R^s, -Z_K+2r_h)\simeq \bH^*(\R_{\geq 0}^s, -Z_K+2r_h).$$
 Indeed, write $s_h=r_h+\Delta_h$ for some $\Delta_h\in L_{\geq 0}$. Then, for  any $l\in L_{\geq 0}$
 $$\chi_{-Z_K+2s_h}(l-\Delta_h)=\chi_{-Z_K+2r_h}(l)-\chi_{-Z_K+2r_h}(\Delta_h).$$
 Therefore, up to a shift $\chi_{-Z_K+2r_h}(\Delta_h)$,  we have the  isomorphisms
 $$ \bH^*(\R^s, -Z_K+2r_h)\simeq \bH^*(\R^s, -Z_K+2s_h), \ \mbox{and} \
  \bH^*(\R_{\geq 0}^s, -Z_K+2r_h)\simeq \bH^*(\R_{\geq 0}^s-\Delta_h, -Z_K+2s_h).$$
  But the contraction which realizes Proposition \ref{9STR2}{\it (c)}
  (which contracts $\R^s$ onto $\R_{\geq 0}^s$ compatible with the weights)
  restricted to $\R_{\geq 0}^s-\Delta_h$, realizes an isomorphism (cf. \cite{LN1})
  $$ \bH^*(\R_{\geq 0}^s-\Delta_h, -Z_K+2s_h)\simeq \bH^*(\R_{\geq 0}^s, -Z_K+2s_h).$$
 Then use these identities together with  Proposition \ref{9STR2}{\it (c)}.
\end{remark}
\section{Preliminaries regarding analytic invariants}\label{ss:NSSAnlattice}
\setcounter{equation}{0}

\subsection{Natural line bundles}\label{ss:UAC}
Fix a complex normal surface singularity $(X,o)$ and in subsections \ref{ss:UAC} and \ref{ss:MultF}
 we assume that the link is  a rational homology sphere.

 By duality, $L'$ is isomorphic to $H^2(\widetilde{X},\bZ)$ and it is the
target of the first Chern class $c_1:{\rm Pic}(\widetilde{X})\to H^2(\widetilde{X},\bZ)$.
This morphism appears in the  exact sequence (induced by the exponential exact sequence
of sheaves):
\begin{equation}\label{eq:PIC}
0\to
 H^1(\tX,\cO_{\tX}) \longrightarrow 
  {\rm Pic}(\tX)\stackrel{c_1}{\longrightarrow} H^2(\tX,\bZ)\to 0.
\end{equation}

In this exact sequence  $c_1$
admits  a natural group section $s_L$  over the integral cycles
$L\subset L'$. Namely, for any $l\in L$ one takes $\cO_{\tX}(l)\in {\rm Pic}(\tX)$
with $c_1(\cO(l))=l$. By \cite{NGr}
$s_L$ can be extended in a unique way  to a natural group section
$s:L'\to {\rm Pic}(\tX)$. Its existence basically is guaranteed by the
facts that $H=L'/L$ is finite, while ${\rm Pic}^0(\tX):=H^1(\tX,\calO_{\tX})$ is torsion free.

\begin{definition}\label{def:natlinebd}  \ix{$\cO_{\tX}(l)$}
The line bundles $s(l')$, indexed by $l'\in L'$, and denoted also by $\cO_{\tX}(l'):=s(l')$,
 will be called {\it natural line bundles}.
\end{definition}

In fact,   a line bundle $\cL\in {\rm Pic}(\tX)$ is natural if and only if
some power of it has the form $\cO_{\tX}(l)$ for an integral cycle $l\in L$.

\bekezdes\label{bek:univabcov}{\bf The universal abelian covering.}
Let $c:(X_a,o)\to (X,o)$ be the universal abelian covering  \ix{$c:(X_a,o)\to (X,o)$}
of $(X,o)$: $(X_a,o)$  is  the  unique normal singular germ such that
$X_a\setminus \{o\}$ is
 the regular covering of $X\setminus \{o\}$ associated with
$\pi_1(X\setminus \{o\})\to H$.

 Since  $\tX\setminus
E\approx X\setminus \{o\}$, $\pi_1(\tX\setminus E)=\pi_1(X\setminus \{o\})\to H$ defines a
regular Galois covering of $\tX\setminus E$ as well. This has a unique
extension $\widetilde{c}:Z\to \tX$ with $Z$ normal and $\widetilde{c}$ finite.
(In other words,  $\widetilde{c}:Z\to \tX$ is the normalized pullback of $c$ via $\phi$.)
The (reduced) branch locus of $\widetilde{c}$ is included in $E$, and the
Galois action of $H$ extends to $Z$ as well.
Since $E$ is a normal
crossing divisor, the only singularities what $Z$ might have are
cyclic quotient singularities. Let
$r:\widetilde{Z}\to Z$ be  a resolution of these singular points
such that $(\widetilde{c}\circ r)^{-1}(E)$ is a normal crossing divisor. Set $p:=\widetilde{c}\circ r$.
\begin{equation}\label{eq:diagramUAC}
\begin{array}{ccccc}
\widetilde{Z} & \stackrel{r}{\longrightarrow} &Z& \stackrel{\psi_a}{\longrightarrow} & (X_a,o) \\
 & &
\Big\downarrow \vcenter{%
\rlap{$\scriptstyle{\widetilde{c}}$}} & &
\Big\downarrow \vcenter{%
\rlap{$\scriptstyle{c}$}} \\
 & &(\tX,E) & \stackrel{\phi}{\longrightarrow}  & (X,o)
\end{array}
 \end{equation}

\begin{theorem}\label{th:3.9}\ \cite{NGr,OkumaRat,Opg}
$\widetilde{c}_*\cO_Z$ is a vector bundle and its $H$-eigensheaf decomposition
 has the form:
 \begin{equation}\label{eq:3.9SUM}
\widetilde{c}_*\cO_Z\simeq \oplus _{\alpha\in \widehat{H}}\cL_\alpha,\end{equation}
where $\cL_{\theta(h)}=\cO_{\tX}(-r_h)$  for any $h\in H$.
In particular, $\widetilde{c}_*\cO_Z\simeq\oplus_{l'\in Q}\cO_{\tX}(-l')$.

More generally, for any $l'\in L'$ one has
 \begin{equation}\label{eq:3.9SUM2}
\widetilde{c}_*\cO_Z(-\widetilde{c}^*(l'))\simeq \oplus _{h\in H}
\cO_{\tX}(-r_h+\lfloor r_h-l'\rfloor).\end{equation}
\end{theorem}

\bekezdes\label{bek:UACPG} {\bf The geometric genus of the universal abelian covering.}
In general (even if $H_1(M,\Q)\not=0$), the geometric genus of $(X,o)$ is defined as $p_g(X,o)=h^1(\tX, \calO_{\tX})$.
It is independent of the resolution.

Assume that the link of $(X,o)$ is a rational homology sphere. In this situation one defines the equivarant geometric genera
(indexed by $H$) as follows.

 Let $(X_a,o)\to (X,o)$ be the universal abelian covering of $(X,o)$,
 and consider the notations of the diagram (\ref{eq:diagramUAC}).
By definition, the  geometric genus $p_g(X_a,o)$ of $(X_a,o)$ is
$h^1(\widetilde{Z},\cO_{\widetilde{Z}})$. Since  $r:\widetilde{Z}\to Z$ is the resolution of the cyclic quotient singularities of $Z$, we have  $p_g(X_a,o)=h^1(\cO_Z)$. Since $\widetilde{c}$ is finite $h^1(\cO_Z)$  equals
$\dim (R^1\widetilde{c}_*\cO_Z)_o$,
 and it  has an eigenspace decomposition
$\oplus_{h\in H}(R^1\widetilde{c}_*\cO_Z)_{o,\theta(h)}$. By Theorem \ref{th:3.9}
the dimension of the $\theta(h)$-eigenspace is
$$p_g(X_a,o)_{\theta(h)}:= \dim\,(R^1\widetilde{c}_*\cO_Z)_{o,\theta(h)}=h^1(\tX,\co_{\tX}(-r_h)).$$
By summation:
$$p_g(X_a,o)=\sum_{h\in H}h^1(\tX,\co_{\tX}(-r_h)).$$
Clearly, for $h=0$ we get $p_g(X_a,o)_{\theta(0)}=p_g(X,o)$.

\begin{definition} If $H_1(M,\Q)=0$ we define the
equivariant geometric genus of $(X,o)$ associated with $h\in H$ by
$ p_g(X_a,o)_{\theta(h)}= h^1(\tX,\co_{\tX}(-r_h))$. Sometimes we abridge  it by $p_{g,h}=p_{g,h}(X,o)$.
\end{definition}

\subsection{Multivariable filtrations and series. Notations} \label{ss:MultF}\cite{CDG,CHR,NJEMS}

\bekezdes{\bf The module $\setZ[[L']]$.}\label{ring}
Once a resolution is fixed, hence the natural basis $\{E_v\}_v$ of $L$ is fixed too,
$\Z[[L]]$ is identified with  $\Z[\bt^{\pm 1}]=\setZ[[t_1^{\pm 1},\ldots,t_s^{\pm 1}]]$.
It is contained in the larger module
$\ringd=\setZ[[t_1^{\pm 1/d},\ldots,t_s^{\pm 1/d}]]$,
the module of formal (Laurent) power series in variables $t_v^{\pm 1/d}$, where $d:=|H|$.
 $\setZ[[L']]$ embeds into $\ringd$ as a  submodule: it  consists of  the $\bZ$-linear
combinations of monomials of type
$$\bt^{l'}=t_1^{l'_1}\cdots t_s^{l'_s}, \ \ \ \mbox{where} \ \
l'=\textstyle{\sum_v\,l'_vE_v}\in L'.$$



\begin{definition}\label{hcompa}
Any series $S(\bt)=\sum_{l'}a_{l'}\bt^{l'}\in \setZ[[L']]$
decomposes in a unique way as \begin{equation}\label{hdec}
S=\sum_{h\in H}S_h,\ \ \mbox{where} \ \
S_h=\sum_{[l']=h}a_{l'}\bt^{l'}.\end{equation} $S_h$ is called the
$h$-component of\, $S$.
\end{definition}
\ix{Poincar\'e series!$h$-component|textbf}

\bekezdes Consider the diagram from (\ref{eq:diagramUAC}) and
set $\phi_a=\psi_a\circ r$ and $p=\widetilde{c}\circ r$. One verifies that
$p^*(l')$ is an {\em integral}  cycle for any $l'\in L'$.

\begin{definition}\label{filtr}
The  $L'$--filtration on the local ring of holomorphic functions
$\cO_{X_a,o}$ is defined as follows.   For any $l'\in L'$, we set \ix{$\cF(l')$ (divisorial filtration)|textbf}
\begin{equation}\label{eq:F1}
\cF(l'):=\{ f\in \cO_{X_a,o}\ | \ {\rm div}(f\circ \phi_a)\geq  p^*(l')\}.
\end{equation}
Notice that the natural action of $H$ on $(X_a,o)$ induces an action on
$\cO_{X_a,o}$, which keeps  $\cF(l')$ invariant. Therefore, $H$ acts
on $\cO_{X_a,o}/\cF(l')$ as well. For any $l'\in L'$, let
$\hh(l')$ be the dimension of the $\theta([l'])$-eigenspace
$(\cO_{X_a,o}/\cF(l'))_{\theta([l'])}$. Then one defines the Hilbert
series \,$H(\bt)$ by  \ix{$\hh(l')$}
\begin{equation}\label{eq:33H}
H(\bt):=\sum_{l'\in L'} \widetilde{\hh}(l')\cdot\bt^{l'} \in \Z[[L']].
\end{equation}
\end{definition}

By \cite{NOSz}, for any $l'\in L'$ there
exists a unique minimal $s(l')\in \cS'$ such that $l'\leq s(l')$ and
$[l']=[s(l')]$.
Since for any $f\in \cO_{X_a,o}$, that part of $div(f\circ \phi_a)$,
which is supported by the exceptional divisor of $\phi_a$, is in the
Lipman's cone of $\widetilde{Z}$, we get
\begin{equation}\label{eq:sl'}
  \cF(l')=\cF(s(l')).
\end{equation}
\bekezdes
For a fixed $l'$ we write $[l']=h$. If
 $l'>0$  one has the exact sequence
\begin{equation}\label{exSeq}
0\to \cO_{\widetilde{Z}}(-p^*(l'))\to
\cO_{\widetilde{Z}}\to \cO_{p^*(l')}\to 0.
\end{equation}
The  $\theta(h)$-eigenspaces form  the exact sequence,
cf. (\ref{eq:3.9SUM2}),
\begin{equation}\label{exSeqphi}
0\to \cO_{\widetilde{X}}(-l')\to \cO_{\widetilde{X}}(-r_h)\to
\cO_{l'-r_h}(-r_h)\to 0.
\end{equation}
In particular, for $l'>0$, 
\begin{equation}\label{eq:hl'}
\widetilde{\hh}(l')=\dim \Big( \frac{H^0(\widetilde{Z},\cO_{\widetilde{Z}})}{
H^0(\widetilde{Z},\cO_{\widetilde{Z}}(-p^*(l')))}\Big)_{\theta(h)}=
\dim\ \frac{H^0(\widetilde{X},\cO_{\widetilde{X}}(-r_h))}
{H^0(\widetilde{X},\cO_{\widetilde{X}}(-l'))}.
\end{equation}
\begin{example}\label{ex: h0}
In (\ref{eq:hl'}) if $l'\in L$ then $r_h=0$.  Hence the $0$-component of $H(\bt)$ is
\begin{equation*}
H_0(\bt)=\sum_{l\in L}\dim \ \Big( \frac{\cO_{X,o}}{\{f\in \cO_{X,o}:
{\rm div}_E(f\circ \phi) \geq l\}}\Big)\,\cdot  \bt^l.
\end{equation*}
This is the Hilbert series of $\cO_{X,o}$ associated
with the divisorial filtration $L\ni l\mapsto \cF_0(l)=\{f\in \cO_{X,o}:
{\rm div}_E(f\circ \phi) \geq l\}$
of all  irreducible exceptional divisors of $\phi$.
\end{example}

\bekezdes  \label{bek:Poinc}
Next, we define the Poincar\'e series $P(\bt)=\sum_{l'\in   \cS'}\pp(l')\bt^{l'}$
 associated with the filtration \ix{$\pp(l')$}\ix{$\hh(l')$}
$\{\cF(l')\}_{l'}$
\begin{equation}\label{eq:4}
P(\bt)=-H(\bt)\cdot \prod_v(1-t_v^{-1}), \ \ \mbox{or} \ \
\pp(l' )=\sum_{I\subset \{1,\ldots, s\}}\, (-1)^{|I|+1}\widetilde{\hh}(l'+E_I),
\end{equation}
where  $E_I=\sum_{v \in I} E_v$.

\subsection{Vanishing theorems, dualities}\label{ss:AnPrel}
 Let $(X,o)$ be a normal surface singularity
(without any restriction regarding its link) and we fix a good resolution $\phi$.
Let $K_{\tX}$ be a canonical divisor on $\tX$, that is, $\Omega^2_{\tX}\simeq \calO_{\tX}(K_{\tX})$.

\begin{theorem}\label{th:GR} {\bf Generalized Grauert--Riemenschneider Theorem.}
\cite{GrRie,Laufer72,Ram,Book}
 Consider a line bundle $\cL\in {\rm Pic}(\widetilde{X})$ such that
$c_1(\cL(Z_K))\in \Delta -\cS_{\Q}$ for some
$\Delta\in L'$ with  $\lfloor \Delta \rfloor =0$.
Then $h^1(l,\cL|_{l})=0$ for any $l\in L_{>0}$.
In particular, $h^1(\widetilde{X},\cL)=0$ too. (Here $\calS_\Q$ denotes  the rational cone generated by $\calS$.)
\end{theorem}

In particular, if
$\calL\in {\rm Pic}(\tX)$ and $l\in L_{>0}$ satisfies $l\in c_1(\calL)+Z_K+\calS$,
then $H^1(\tX, \calL)=H^1(l, \calL|_l)$.

 As above, we denote by $\lfloor Z_K\rfloor $ the integral part of $Z_K$, and by
$\lfloor Z_K\rfloor_+ $ its effective part. The above statements imply the following.
If $\lfloor Z_K\rfloor_+=0$ then $p_g=0$.
If $\lfloor Z_K\rfloor_+>0$ then for
 any $Z\geq \lfloor Z_K\rfloor_+$, $Z\in L$,
$p_g=h^1(\cO_{Z})$.

Furthermore,  if $l\in \cS$ and $n\in\Z_{\geq 0}$ such that $nl+\lfloor Z_K\rfloor>0$ then
by the above vanishing theorem we have  $H^1(\tX, \calO_{\tX}(-\lfloor Z_K\rfloor-nl-s_h))=0$, hence
\begin{equation*}
\dim \frac{H^0(\cO_{\tX}(-s_h))}{H^0(\cO_{\tX}(-\lfloor Z_K\rfloor-nl-s_h))}=
\chi(\lfloor Z_K\rfloor+nl)- (s_h, \lfloor Z_K\rfloor +nl)+h^1(\calO_{\tX}(-s_h)).
\end{equation*}
This implies that  for any $l\in \calS\setminus \{0\}$ and $n\gg 0$, and $l'_h$ either $r_h$ or $s_h$ we have
\begin{equation}\label{eq:vanh0}
\dim \frac{H^0(\cO_{\tX}(-l'_h))}{H^0(\cO_{\tX}(-nl-l'_h))}=
-\frac{n^2l^2}{2}+\mbox{lower order terms in $n$}.
\end{equation}


For certain cycles the Grauert-Riemenschneider Theorem \ref{th:GR} can be improved.

\begin{proposition}\label{prop:VAN0}\ {\bf Lipman's Vanishing Theorem.} \cite[Theorem 11.1]{Lipman}, \cite{Book}\
Take  $l\in L_{>0}$ with  $h^1(\cO_l)=0$  and $\cL\in {\rm Pic}(\tX)$ for which
 $(c_1\cL,E_v)\geq 0$ for any $E_v$ in the support of\, $l$. Then  $h^1(l,\cL)=0$.
 \ix{Vanishing Theorem!Lipman's|textbf}
\end{proposition}
\bekezdes By {\bf Serre duality} $H^i(l, \calL)=H^{1-i}(l, \calL^{-1}(K_{\tX}+l))^*$
for any  $l\in L_{>0}$, $\calL\in {\rm Pic}(\tX)$ and  $i=0,1$.

\bekezdes\label{bek:LauferDual} {\bf Laufer's Duality.} \cite{Laufer72},
\cite[p. 1281]{Laufer77}
We can identify the dual space $H^1(\tX,\cO_{\tX})^*$ with the space of global holomorphic
2-forms on $\tX\setminus E$ up to the subspace of those forms which can be extended
holomorphically over $\tX$:
$H^1(\tX,\cO_{\tX})^*\simeq 
H^0(\tX\setminus E,\Omega^2_{\tX})/ H^0(\tX,\Omega^2_{\tX})$. 
Here $H^0(\tX\setminus E,\Omega^2_{\tX})$ can be replaced by
$H^0(\tX,\Omega^2_{\tX}(Z))$ for
any $Z>0$ with  $h^1(\cO_Z)=p_g$. Indeed, for any  $Z>0$, from the exacts sequence of sheaves
$0\to\Omega^2_{\tX}\to \Omega^2_{\tX}(Z)\to \calO_{Z}(Z+K_{\tX})\to 0$
and from the vanishing
$h^1(\Omega^2_{\tX})=0$ and Serre duality
\begin{equation}\label{eq:duality}
H^0(\Omega^2_{\tX}(Z))/H^0(\Omega^2_{\tX})=H^0(\calO_Z(Z+K_{\tX}))\simeq H^1(\calO_Z)^*.
\end{equation}
If  $H^1(\calO_Z)\simeq H^1(\calO_{\tX})$ then   the inclusion
$H^0(\Omega^2_{\tX}(Z))/H^0( \Omega^2_{\tX})\hookrightarrow
H^0(\tX\setminus E, \Omega^2_{\tX})/H^0(\Omega^2_{\tX})$
is an isomorphism.

\subsection{Cohomological cycles}
\label{rem:antmatroid}  \cite[4.8]{MR} Assume that $p_g>0$. The set
$L_{p_g}:=\{l\in L_{>0}\, : \, h^1(\cO_l)=p_g\}$ has a unique minimal element, denoted by $Z_{coh}$, and called
the {\it cohomological cycle} of $\phi$. It has the property that
 $h^1(\cO_l)<p_g$
for any $l\not\geq Z_{coh}$ ($l>0$).
By the consequences of Theorem \ref{th:GR} we obtain that $Z_{coh}\leq  \lfloor Z_K\rfloor_+$.
If $p_g=0$ then we set $Z_{coh}:=0$ by definition.
More generally, we have the following results.

\begin{proposition}\label{prop:cohcyc2} Fix a line bundle  $\cL\in {\rm Pic}(\widetilde{X})$.

(a) Assume that  $h^1(\widetilde{X},\cL)>0$.  The set
$L_{\cL}:=\{l\in L_{>0}\, : \, h^1(l,\cL)=h^1(\widetilde{X},\cL)\}$
has a unique minimal element, denoted by $Z_{coh}(\cL)$, called
the cohomological cycle of $\cL$ (and of $\phi$). It has the property that
$  h^1(l,\cL)<h^1(\widetilde{X},\cL)$
for any $l\not\geq Z_{coh}(\cL)$ ($l>0$).

(b) Let $l_1, l_2\in L_{>0}$ be effective cycles, and set
$l=\min\{l_1, l_2\}$ and $\overline{l}=\max\{l_1,l_2\}$. Then
$$h^1(\overline{l},\cL)+h^1(l,\cL)\geq h^1(l_1,\cL)+h^1(l_2,\cL).$$
We will refer
to this inequality as  the {\it `opposite' matroid rank inequality} of $h^1(\cL)$.

(c)
In particular, for any $l\in L_{>0}$ we have $h^1(l,\cL)=h^1(\min \{l,Z_{coh}(\cL)\},\cL)$.
\end{proposition}

\begin{proof} In {\it (b)}
we can assume that  $a_i=l_i-l>0$, $i=1,2$.
 Consider the diagram with exact rows and columns.
$$\begin{array}{ccccccc}
 \ & \ & H^1(\cL(-l_1)|_{a_2})& \to & H^1(\cL(-l)|_{a_2}) & \to & 0 \\
\ & \ & \downarrow & \ & \downarrow & & \\
 H^1(\cL(-l_2)|_{a_1}) & \to& H^1(\cL|_{\bar{l}})& \to & H^1(\cL|_{l_2}) & \to & 0\\
 \downarrow & \ & \downarrow & \ & \downarrow & & \\
 H^1(\cL(-l)|_{a_1}) & \to& H^1(\cL|_{l_1})& \to & H^1(\cL|_{l}) & \to & 0\\
 \downarrow & \ & \downarrow & \ & \downarrow & & \\
 0 & \ & 0 & \ & 0 & \ & \
 \end{array}$$
 The exactness of the first row follows from the exact sequence
 $\cL(-l-a_1)|_{a_2}\to \cL(-l)|_{a_2}\to \cL_{a_1\cap a_2}\to 0$, where the support of
 $\cL_{a_1\cap a_2}$ is 0--dimensional.
 From the diagram
 one gets that \begin{equation}\label{eq:matr}
H^1(\cL|_{\bar{l}})\to H^1(\cL|_{l_1})\oplus H^1(\cL|_{l_2})\to H^1(\cL|_{l})\to 0\end{equation}
 is exact, hence {\it (b)} follows.

Assume that $h^1(\cL|_{l_1})=h^1(\cL|_{l_2})=h^1(\tX,\cL)$ for $l_1\not = l_2$,  $l_1, l_2\in L_{>0}$.
Set $l=\min\{l_1,l_2\}$. If $l=0$ then there is an exact sequence $0\to \cL|_{l_1+l_2}\to
\cL|_{l_1}\oplus \cL|_{l_2}\to A \to 0$, where   $A$ has zero--dimensional support, hence
$H^1(\cL|_{l_1+l_2})  \to H^1(\cL|_{l_1})\oplus H^1(\cL|_{l_2})=\C^{2h^1(\cL)}$ surjective, a fact which cannot happen.
Hence $l\not=0$. Then (\ref{eq:matr})
implies $H^1(\cL|_{l})=h^1(\cL)$ too. Hence, whenever $l_1,l_2\in L_{\cL}$ one also has
 $\min\{l_1,l_2\}\in L_{\cL}$. This implies {\it (a)}.  Finally, {\it (a)} and {\it (b)} implies {\it (c)}.
\end{proof}

%

If $h^1(\widetilde{X},\cL)=0$ then we define $Z_{coh}(\cL):=0$.

\section{The analytic lattice cohomology of $(X,o)$}\label{ss:anphi}

\subsection{Defintion and independence  of the choice of the  rectangle}\label{ss:anR}

\bekezdes
Our goal is to construct the {\it analytic lattice cohomology } of a normal surface singularity $(X,o)$
under the assumption that the link is a rational homology sphere.
In particular, for any  spin$^c$--structures of the link, or  for any representative $[k]\in{\rm Char}/2L$,
we wish to define a graded $Z[U]$--module.

We fix a good resolution $\phi$ and $h\in H$. Write $Z_{coh,h}$ for $Z_{coh}(\calO_{\tX}(-r_h))$.

For any $c\in L$, $c\geq Z_{coh,h}$,  we consider the rectangle $R(0,c)=\{l\in L\,:\,
0\leq l\leq c\}$. By definition of $Z_{coh,h}$
\begin{equation}\label{eq:h^1}
p_{g,h}=h^1(\tX,\calO_{\tX}(-r_h))=h^1(c,\calO_{\tX}(-r_h)).\end{equation}
Here we might  consider the $c=\infty$ case too, in this case $R(0,c)=L_{\geq 0}$.

\bekezdes {\bf The weight function.}
We consider the multivariable  Hilbert function $\widetilde{h}$, cf. (\ref{eq:hl'}), and
$$\hh:R(0,c)\to\Z, \ \  \hh(l):=\widetilde{\hh}(l+r_h)=\dim \, \big( H^0(\calO_{\tX}(-r_h))/H^0(\calO_{\tX}(-l-r_h))\,\big)$$
associated with the
divisorial filtration of $\calO_{X_a,o}$ and the resolution $\phi$, cf. \ref{eq:hl'}.
Clearly $\hh$ is increasing (that is, $\hh(l_1)\geq \hh(l_2)$ whenever $l_1\geq l_2$) and  $\hh(0)=0$.
Next, for any $l\in R(0,c)$, we consider the function
$$\hh^\circ(l)=p_{g,h}-h^1(\calO_l(-r_h)),$$ where $h^1(\calO_{l=0}(-r_h))$, by definition, is 0.
Then  $\hh^\circ$ is decreasing,
 $\hh^\circ (0)=p_{g,h}$  and $\hh^\circ (c)=0$, cf. (\ref{eq:h^1}).
We have the following reinterpretation in terms of (twisted) 2--forms.
For any $\bar{l}\geq 0$  consider the exact sequence
$$0\to \Omega^2_{\tX}(r_h)\to  \Omega^2_{\tX}(r_h+\bar{l})\to \Omega^2_{\tX}(r_h+\bar{l})|_{\bar{l}}
\to 0.$$ Since $H^1(\Omega^2_{\tX}(r_h))=0$ (cf. Theorem \ref{th:GR}) for any $\bar{l}\geq 0$ we obtain
(using  Serre duality too)
\begin{equation}\label{eq:Omega}
\frac{H^0(\Omega^2_{\tX}(r_h+\bar{l}))}{H^0(\Omega^2_{\tX}(r_h))}=H^0(
\bar{l}, \Omega^2_{\tX}(r_h+\bar{l}))\simeq H^1(\calO_{\bar{l}}(-r_h))^*.
\end{equation}
This applied for\, $\bar{l}=c$ and  $\bar{l}=l$ gives
\begin{equation}\label{eq:hcirc}
\dim\, \frac{H^0(\tX,\Omega^2_{\tX}(c+r_h))}{H^0(\Omega^2_{\tX}(l+r_h))}=p_{g,h}-h^1(\calO_l(-r_h))=\hh^\circ (l).
\end{equation}

\bekezdes {\bf The lattice cohomology.}
We consider the natural cube-decomposition of $R(0,c)$
 (where the 0-cubes  are the lattice points)  and
the set of cubes $\{\calQ_q\}_{q\geq 0}$ of $R(0,c)$ as in \ref{9complex}.
Then we define the weight function
\begin{equation}\label{eq:wean}
w_0:\calQ_0\to \Z, \ \ \  w_0(l)=\hh(l)+\hh^\circ (l)-\hh^\circ (0)=\hh(l)-h^1(\calO_l(-r_h)).
\end{equation}
Clearly,   $w_0(0)=0$. Let us list some properties of $w_0$.

First of all, note that $0\leq \hh^\circ(l)\leq p_{g,h}$ for every $l$, hence when $c=\infty$ then
$\hh$ and $w_0$ have comparable  asymptotic behaviours for $l\gg 0$. Using the monotonity of $\hh$,
(\ref{eq:sl'})   and
(\ref{eq:vanh0})
 a computation shows that  $w_0$ satisfies the requirement \ref{9weight}(a), namely,
 $w_0^{-1}((\-\infty, n]) \ \ \mbox{is finite for any $n\in\Z$}$.

 Next,  since $\hh$ is induced by a filtration, it satisfies the
matroid rank inequality
$\hh(l_1)+\hh(l_2)\geq  \hh(\overline{l})+\hh(l)$,
where $l=\min\{l_1, l_2\}$ and $\overline{l}=\max\{l_1,l_2\}$.
On the other hand, $h^1$ satisfies the `opposite' matroid rank inequality, see
\ref{rem:antmatroid}. Therefore, $w_0$ itself satisfies the matroid rank inequality (where $l_1,l_2\geq 0$)
  \begin{equation}\label{eq:matroidw00}
 w_0(l_1)+w_0(l_2)\geq  w_0(\overline{l})+w_0(l).
 \end{equation}

Furthermore, similarly as in \ref{9dEF1}, we define
$w_q:\calQ_q\to \Z$ by $ w_q(\square_q)=\max\{w_0(l)\,:\, l \
 \mbox{\,is any vertex of $\square_q$}\}$.
 In the sequel  we write $w$ for the system $\{w_q\}_q$ if there is no confusion.
 The weight functions $\{w_q\}_q$ define the lattice cohomology $\bH^*(R(0,c),w)$
 and the graded root $\RR(R(0,c),w)$
associated with $R(0,c)$ and $w$.

\begin{lemma}\label{lem:INDEPAN}
$\bH^*(R(0,c),w)$ and  $\RR(R(0,c),w)$ are independent on the choice of $c\geq Z_{coh,h}$.
\end{lemma}
\begin{proof}
Fix some $c\geq Z_{coh,h}$ and choose $E_v\subset |c- Z_{coh,h}|$.
Then for any $l\in R(0,c)$ with
$l_v=c_v$ we have $\min\{l, Z_{coh,h}\}=\min\{l-E_v, Z_{coh,h}\}$.
Therefore, by \ref{rem:antmatroid}, $h^1(\calO_{l-E_v}(-r_h))=h^1(\calO_l(-r_h))$, thus
$w_0(l-E_v)\leq w_0(l)$. Then for any $n\in \Z$, a strong deformation retract in the direction $E_v$ realizes
a homotopy equivalence between the spaces $S_n\cap R(0,c)$ and  $S_n\cap R(0,c-E_v)$.
A natural retract $r:S_n\cap R(0,c)\to S_n\cap R(0,c-E_v)$ can be defined as follows (for notation see \ref{9complex}).
If $\square =(l,I)$ belongs to $ S_n\cap R(0,c-E_v)$ then $r$ on $\square$ is defined as the identity.
If $(l,I)\cap  R(0,c-E_v)=\emptyset$, then $l_v=c_v$, and  we set  $r(x)=x-E_v$. Else,
$\square =(l,I)$ satisfies $v\in I$ and $l_v=c_v-1$. Then we retract $(l,I)$ to $(l, I\setminus v)$ in the $v$--direction.
The strong deformation retract is defined similarly.
\end{proof}

\begin{corollary}\label{cor:veges}
(a) The graded root $\RR(R(0,c),w)$ satisfies $|\chic^{-1}(n)|=1$ for any $n\gg 0$.

(b) $\bH^*_{red}(R(0,c),w)$ is a finitely generated $\Z$-module (for any finite or infinite $c\geq Z_{coh}$).
\end{corollary}
\begin{proof}
For any $n\gg 0$ we have $R(0,c)=S_n$, hence $S_n$ is contractible for such $n$.
\end{proof}

\subsection{Independence  of $\phi$}  Rewrite the $c$--independent module   $\bH^*(R(0,c),w)$ as $\bH^*_{an,h}(\phi)$, and the garded root as $\RR_{an, h}(\phi)$.
\begin{theorem}\label{th:annlattinda} The graded $\Z[U]$--module
$\bH^*_{an,h}(\phi)$  and the graded root  $\RR_{an, h}(\phi)$
are  independent of the choice of the resolution $\phi$.
\end{theorem}
\begin{proof}
We need to verify that $\bH^*_{an,h}(\phi)$ and $\RR_{an, h}(\phi)$ are
 stable with respect to blow up of a point. We
discuss two cases according to the position of the point with the singular locus of  $E$.

\underline{ {\bf Case A.}} We fix a resolution $\phi$, and denote the blow up of a point of $E_{v_0}\setminus \cup_{w\not= v_0}E_w$
by $\pi$, and set $\phi':= \phi\circ \pi$.
Let $\Gamma$ and $\Gamma'$ be the corresponding graphs, $L(\Gamma),\ L(\Gamma')$  the lattices
and  $\, (\,,\,), \ (\,,\,)'$ the  intersection forms.

 We denote the new
$(-1)$-vertex of $\Gamma'$ by $E_{new}$.  In our notations we identify  $E_v\in L$
with its strict transform in $L(\Gamma')$. We have the next natural morphisms:
 $\pi_*:L(\Gamma')\to
L(\Gamma)$ defined by $\pi_*(\sum x_vE_v+x_{new}E_{new})=\sum
x_vE_v$, and $\pi^*:L(\Gamma)\to L(\Gamma') $  defined by
$\pi^*(\sum x_vE_v)=\sum x_vE_v +x_{v_0}E_{new}$. They can be extended by similar formulae to
rational cycles too, and $\pi^*(L'(\Gamma)\subset L'(\Gamma')$.  They satisfy the `projection formula'
$(\pi^*x,x')'=(x,\pi_*x')$.  This shows that
$(\pi^*x,\pi^*y)'=(x,y)$ and $(\pi^*x,E_{new})'=0$ for any
$x,y\in L'(\Gamma)$.
Associated with $\phi$,  let $\hh$, $\hh^\circ$  be the functions defined above,
  $w_0$ the analytic weight and
 $S_n(\phi)=\cup\{\square\,:\, w(\square)\leq n\}$. We use similar notations  $\hh'$, $(\hh^{\circ})'$, $w_0'$ and
 $S_n(\phi')$ for  $\phi'$. Let also $r_h\in L'(\Gamma)$ and $r_h'\in L'(\Gamma')$ be the universal cycles associated with $h\in H$.

\begin{lemma}\label{lem:r_h} $\pi^*(r_h)=r_h'$.\end{lemma}
\begin{proof}
The composition $\varphi_{\tX}: {\rm Div}(\tX)\to {\rm Pic}(\tX)\stackrel{c_1}{\longrightarrow} L'\to L'/L
=H$ is realized by $D\mapsto [D\cap \partial \tX]$ (for $\tX$ conveniently small and
$\partial \tX=M$). If $D'\in{\rm Div(\tX'})$
is the strict transform of $D\in{\rm Div}(\tX)$ then $\varphi_{\tX'}(D')=\varphi_{\tX}(D)$ in $H$.
Therefore, if we chose $x\in L'(\Gamma)$ and $x'\in L'(\Gamma')$ such that $D+x$ and $D'+x'$ are
numerically trivial in $H_1(\tX,\partial \tX,\Q)$ (i.e.
 $(D+x,E_v)_{\tX}=0$ for all $v\in\calv$, and similarly for $D'+x'$) then $x'=\pi^*x$. Hence, in the two resolutions,
$x\in L'(\Gamma)$ and $\pi^*x\in L'(\Gamma')$ have the  same class  in $H$. On the other hand, clearly,
 all the $E_v$--entries of
$\pi^*r_h$ are in $[0,1)$.
\end{proof}
\begin{lemma}\label{lem:Leray}
$H^*(\tX', \pi^*\calL)=H^*(\tX, \calL)$ and  $H^*(\pi^*x, \pi^*\calL)=H^*(x, \calL)$
 for any line bundle $\calL\in {\rm Pic}(\tX)$  and $x\in L(\Gamma)$.
\end{lemma}
\begin{proof}
The first identity follows from Leray spectral sequence, the second one from the first via
exact sequences of type $0\to \calL(-x)\to\calL\to \calL|_x\to 0$.
\end{proof}

\bekezdes \label{bek:1}  For $a\leq 0$ and $x\in R$  we claim that
    $H^0(\tX', \calO_{\tX'}(-\pi^*x-\pi^*r_h-aE_{new}))=H^0(\tX', \calO_{\tX'}(-\pi^*x-\pi^*r_h))$.
   Indeed, take the exact sequence of sheaves
   $$0\to \calO_{\tX'}(-\pi^*x-\pi^*r_h)\to \calO_{\tX'}(-\pi^*x-\pi^*r_h-aE_{new})\to \calO_{-aE_{new}}(-\pi^*x-\pi^*r_h -aE_{new})\to 0$$
   and use that $h^0(\calO_l(l)\otimes \calL)=0$ for any $l>0$ and line bundle $\calL$ with
    $(c_1\calL,E_v)=0$ for any $E_v\in |l|$.
   This last vanishing follows from the Grauert--Riemenschneider Theorem via Serre duality.
Therefore (using Lemma \ref{lem:Leray} too)  $\hh'(\pi^*x+aE_{new})$ equals
$$\dim \frac{H^0(\calO_{\tX'}(-\pi^*r_h))}{
H^0(\calO_{\tX'}(-\pi^*x-\pi^*r_h-aE_{new}))}=
\dim \frac{H^0(\calO_{\tX'}(-\pi^*r_h))}{
H^0(\calO_{\tX'}(-\pi^*(x+r_h)))}=
\dim \frac{H^0(\calO_{\tX}(-r_h))}{
H^0(\calO_{\tX}(-x-r_h))}=\hh(x).$$
Hence
\begin{equation}\label{eq:MON1}
 \hh'(\pi^*x+aE_{new}) \ \left\{ \begin{array}{l}
 = \hh(x) \ \mbox{ for any $a\leq 0$} \\
 \mbox{is increasing for $a\geq 0$}.\end{array}\right.
\end{equation}
\bekezdes \label{bek:2}
Using the exact sequence $$0\to \calO_{aE_{new}}(-\pi^*x-\pi^*r_h)\to \calO_{\pi^*x+aE_{new}}(-r_h')\to \calO_{\pi^*x}(-r_h')\to 0$$ and
Lipman's vanishing $h^1(\calO_{aE_{new}}(-\pi^*x-\pi^*r_h))=0$ from \ref{prop:VAN0}, we get that
$h^1(\calO_{\pi^*x+aE_{new}}(-r_h'))=h^1(\calO_{\pi^*x}(-r_h'))$ for any $a\geq 0$. Furthermore, from
$$0\to \calO_{E_{new}}(-\pi^*r_h-\pi^*x +E_{new})\to \calO_{\pi^*x}(-r_h')\to \calO_{\pi^*x-E_{new}}(-r_h')\to 0$$
we get that $h^1(\calO_{\pi^*x-E_{new}}(-r_h'))=h^1(\calO_{\pi^*x}(-r_h'))$ too. On the other hand,
since $\pi^*(r_h)=r_h'$, by Lemma \ref{lem:Leray},
$h^1(\calO_{\pi^*x}(-r_h'))=h^1(\calO_x(-r_h))$. Therefore,
 \begin{equation}\label{eq:MON2}
 h^1(\calO_{\pi^*x+aE_{new}}(-r_h')) \ \left\{ \begin{array}{l}
  \mbox{is increasing  for $a\leq  -1$}, \\
= h^1(\calO_x(-r_h)) \ \mbox{ for any $a\geq  -1$}. \\
\end{array}\right.
\end{equation}
These combined provide
\begin{equation}\label{eq:HOM2a}
a\mapsto w_0'(\pi^*x+aE_{new}) \ \left\{ \begin{array}{l}
\mbox{is  decreasing for $a\leq  -1$},\\
 = w_0(x) \ \mbox{ for  $a= -1$ and $a=0$,} \\
 \mbox{is increasing for $a\geq 0$}.\end{array}\right.
\end{equation}
Recall that we can compute $\bH^*_{an,h}(\phi)$ using the cube
$R(0,c)$ with $c\geq Z_{coh,h}(\phi)$. By Lemma \ref{lem:Leray} we obtain that
$\pi^*c\geq Z_{coh,h}(\phi')$,
hence $\bH^*_{an,h}(\phi')$ can be computed in $R(0, \pi^*c)$.
But we can take $c=\infty$ as well.

Furthermore, if $w_0'(\pi^*x+aE_{new})\leq n$, then
$w_0(x)\leq n$ too. In particular, the projection $\pi_{\R}$ in the direction of $E_{new}$
induces a well-defined map $\pi_{\R}:S_n(\phi')\to S_n(\phi)$.
We claim that this is a   homotopy equivalence  (with all fibers non-empty and contractible).
\bekezdes \label{bek:proof1}
We proceed in two steps.
First we prove that $\pi_{\R}:S_n(\phi')\to S_n(\phi)$ is onto.


Consider a zero dimensional cube (i.e. lattice point) $x\in S_n(\phi)$. Then $w_0(x)\leq n$. But then $w_0'(\pi^*x)
=w_0(x)\leq n$ too, hence $\pi^*(x)\in S_n(\phi')$ and $x=\pi_{\R}(\pi^*x)\in {\rm im}(\pi_{\R})$.

Next, take  a cube $(x,I)\subset  S_n(\phi)$ ($I\subset \cV$). This means that
$w_0(x+E_{I'})\leq n$ for any $I'\subset I$. But
\begin{equation}\label{eq:eps}
\pi^*(x+E_{I'})=\pi^*x+ E_{I'}+\epsilon\cdot E_{new},
\end{equation}
where $\epsilon=0$ if $v_0\not \in I'$ and $\epsilon =1$ otherwise. Hence
\begin{equation}\label{eq:eps2}
w_0'(\pi^*x+E_{I'}) =w_0'(\pi^*(x+E_{I'})-\epsilon E_{new})\stackrel {(\ref{eq:HOM2a})}{=}
w_0(x+E_{I'})\leq n.
\end{equation}
Therefore $(\pi^*x,I)\in S_n(\phi')$ and $\pi_{\R}$ projects $(\pi^*x, I)$  isomorphically  onto $(x,I)$.

Next, we show that $\pi_{\R}$ is in fact a homotopy equivalence.
In order to prove this fact it is enough to verify that if
$\square\in S_{n}(\phi)$ and $\square ^\circ$ denotes its relative interior,
then $\pi_{\R}^{-1}(\square^\circ) \cap S_{n}(\phi')$ is contractible.

Let us start again with a lattice point $x\in S_n(\phi)$. Then $\pi_{\R}^{-1}(x)\cap S_n(\phi')$ is a real  interval
(whose end-points are lattice points, considered in the real line of the $E_{new}$ coordinate).
Let us denote it by $\cI(x)$. Now, if $\square=(x,I)$, then we have to show that
all the intervals $\cI(x+E_{I'})$ associated with all the subsets $I'\subset I$ have a common
lattice point. But this is exactly what we verified above: the $E_{new}$ coordinate of $\pi^*(x)$ is such a common
point. Therefore, $\pi_{\R}^{-1}(\square ^\circ)\cap S_n(\phi')$
has a deformation retract (in the $E_{new}$ direction)
to $(\pi^*x, I)^\circ$.

For any $l\in L$ let  $N(l)\subset \R^s$ denote the union of all cubes which have $l$ as one of their vertices.
Let $U(l)$ be its interior. Write $U_n(l):=U(l)\cap S_n(\phi)$. If $l\in S_n(\phi)$ then $U_n(l)$ is a contractible
neighbourhood of $l$ in $S_n(\phi)$. Also,  $S_n(\phi)$ is covered by $\{U_n(l)\}_l$.
Moreover, $\pi_{\R}^{-1}(U_n(l))$ has the homotopy type of $\pi_{\R}^{-1}(l)$, hence it is contractible.
More generally, for any cube $\square$,
$$\pi_{\R}^{-1}(\cap _{\mbox{$v$ vertex of $\square$}} U_n(l)) \sim \pi_{\R}^{-1}(\square ^\circ)$$
which is contractible by the above discussion. Since all the intersections of $U_n(l)$'s are of these type,
we get that the inverse image of any intersection is contractible. Hence by \v{C}ech covering
(or Leray spectral sequence) argument, $\pi_{\R}$
induces an isomorphism
$H^*(S_n(\phi'),\Z)=H^*(S_n(\phi),\Z)$. In fact, this already shows that $\bH^*_{an,h}(\phi')=\bH^*_{an,h}(\phi)$.
By the identification of the connected components of $S_n(\phi)$ and $S_n(\phi')$ we also have
$\RR_{an,h}(\phi')=\RR_{an,h}(\phi)$. Note that compatibility with the $U$--action also follows from the corresponding inclusions of the $S_n$--spaces.

In order to prove the homotopy equivalence,
one can use   quasifibration, defined in  \cite{DoldThom};
 see also \cite{DadNem}, e.g. the relevant Theorem 6.1.5.
Since
$\pi_{\R}:S_{n}(\phi')\to  S_{n}(\phi)$  is a quasifibration,
and  all the fibers are contractible, the homotopy equivalence follows.

\bekezdes \label{bek:proof2} \underline{{\bf Case B.}} Assume that
we blow up an intersection point $E_{v_0}\cap E_{v_1}$.
The proof  starts very similarly, however at some point
there are two  major differences, hence  we need several additional arguments.

With very similar notation, in this case we have (define) $\pi^*(\sum_vx_vE_v)=\sum_v x_vE_v+(x_{v_0}+
x_{v_1})E_{new}$. Then the strategy is the same as above in Case A, but two
differences appear: the first one is related with $\pi^*r_h$:  Lemma \ref{lem:r_h} is not always true.
The second one is  related with $\pi^*E_{I'}$ in (\ref{eq:eps}).

Let us analyse the analogue of Lemma \ref{lem:r_h}. By the very same proof we have the following
\begin{lemma}\label{lem:r_h2}
Write $r_h$ as $\sum_va_vE_v$ for some $a_v\in[0,1)$. Then $r_h'=\pi^*r_v$ if and only if $a_{v_0}+a_{v_1}<1$. Otherwise
$r_h'=\pi^*r_v-E_{new}$.
\end{lemma}
We divide the proof of Case B in two parts, according to the two cases of Lemma \ref{lem:r_h2}.

\vspace{2mm}

\underline{{\bf Case B.I.}} Assume that $r_h'=\pi^*r_h$.

Then all the statements of Case A  from \ref{bek:1} and \ref{bek:2}  remain valid (including
the key (\ref{eq:HOM2a})).
However, \ref{bek:proof1} should be modified.
The modifications start   in (\ref{eq:eps}). Indeed, in this case
\begin{equation}\label{eq:eps3}
\pi^*(x+E_{I'})=\pi^*x+ E_{I'}+\epsilon\cdot E_{new},
\end{equation}
where $\epsilon$ is the cardinality of $I'\cap \{v_0,v_1\}$. This can be 0, 1 or 2. Therefore, if
$\{v_0,v_1\}\not\subset I$, then $\epsilon \in \{0,1\} $ for any $I'$, hence for such cubes $(x,I)$
all the arguments of \ref{bek:proof1} work.

\bekezdes\label{bek:3}
Assume in the sequel that $\{v_0, v_1\}\subset I$. Write $J=I\setminus \{v_0,v_1\}$.

 There are two cube--candidates  of $L(\Gamma')\otimes \R$ which might
cover the cube $(x,I)\in S_n(\phi)$. One of them is $(\pi^*x,I)$ (as above).  However,  by  (\ref{eq:HOM2a})
the lattice points $\pi^*(x+E_I)=\pi^*x+E_I+2E_{new}$ and $\pi^*(x+E_I)-E_{new}=\pi^*x+E_I+E_{new}$ are in
$S_n(\phi')$, but the vertex $\pi^*x+E_I$ of $(\pi^*x,I)$ is not necessarily in $S_n(\phi')$.

Another candidate is $(\pi^*x+E_{new},I)$, but here again
$\pi^*x$ and $\pi^*x-E_{new}$ are in $S_n(\phi')$  but $\pi^*x+E_{new}$ might be not.
So both cubes a priori are obstructed if we apply  merely (\ref{eq:HOM2a}).

Next we analyze these obstructions with more details and we show that one of the candidate cubes works.

\bekezdes\label{bek:A} {\bf Case 1.}
 Assume that $w_0'(\pi^*x)=w_0'(\pi^*x+E_{new})$. Then by (\ref{eq:MON1}) and (\ref{eq:MON2})
 we obtain that $\hh'(\pi^*x)=\hh'(\pi^*x+E_{new})$.  By the matroid rank inequality of $\hh'$
 we get that $\hh'(\pi^*x+E_{J'})=\hh'(\pi^*x+E_{J'}+E_{new})$ for any $J'\subset J$.
 This again via (\ref{eq:MON1}) and (\ref{eq:MON2})
 shows that $w_0'(\pi^*x+E_{J'})=w_0'(\pi^*x+E_{J'}+E_{new})$.
 In particular,
 $$w_0'(\pi^*x+E_{J'}+E_{new})=w_0'(\pi^*x+E_{J'})=w_0'(\pi^*(x+E_{J'}))=w_0(x+E_{J'})\leq n.$$
That is, the  vertices of type $\pi^*x+E_{J'}+E_{new}$ of $(\pi^*x+E_{new},I)$
are in $S_n(\phi')$. For all other vertices we already know this fact (use (\ref{eq:HOM2a})).
Hence $(\pi^*x+E_{new},I)$ is in $S_n(\phi')$ and it projects via $\pi_{\R}$ bijectively to $(x,I)$.
Furthermore, $\pi_{\R}^{-1}(x,I)^\circ \cap S_n(\phi')$ admits a deformation retract to
 $(\pi^*x+E_{new},I)^\circ $, hence it is contractible.

\bekezdes \label{bek:B} {\bf Case 2.}
 Assume that $w_0'(\pi^*x+E_I)=w_0'(\pi^*x+E_I+E_{new})$,
 or $w_0'(\pi^*(x+E_I)-2E_{new})=w_0'(\pi^*(x+E_I)-E_{new})$.
  Then by (\ref{eq:MON1}) and (\ref{eq:MON2})
 we obtain that $h^1(\calO_{\pi^*x+E_I}(-r_h'))=h^1(\calO_{\pi^*x+E_I+E_{new}}(-r_h'))$.  By the opposite
 matroid rank inequality of $h^1(\calO_{\tX'}(-r_h'))$ and (\ref{eq:MON1}) and (\ref{eq:MON2}) again we obtain that
 $w_0'(\pi^*x+E_I-E_{J'})=w_0'(\pi^*x+E_I-E_{J'}+E_{new})$.
 In particular,
 $$w_0'(\pi^*x+E_I-E_{J'})=w_0'(\pi^*x+E_I-E_{J'}+E_{new})=w_0'(\pi^*(x+E_I-E_{J'})-E_{new})=w_0(x+E_I-E_{J'})\leq n.$$
That is, the  vertices of type $\pi^*x+E_I-E_{J'}$ of $(\pi^*x,I)$
are in $S_n(\phi')$. For all other vertices we already know this fact (use (\ref{eq:HOM2a})).
Hence $(\pi^*x,I)$ is in $S_n(\phi')$ and it projects via $\pi_{\R}$ bijectively to $(x,I)$.
Furthermore, $\pi_{\R}^{-1}(x,I)^\circ \cap S_n(\phi')$ admits a deformation retract to
 $(\pi^*x,I)^\circ $, hence it is contractible.

\bekezdes \label{bek:C} {\bf Case 3.} Assume that the assumptions from {\bf Case 1} and {\bf Case 2} do not hold.
This means that
$$\left\{ \begin{array}{l}
\hh'(\pi^*x)<\hh'(\pi^*x+E_{new}), \ \mbox{and} \\
h^1(\calO_{\pi^*x+E_I}(-r_h'))< h^1(\calO_{\pi^*x +E_I+E_{new}}(-r_h')).\end{array}\right.$$
This reads as follows (cf. (\ref{eq:duality})
$$\left\{ \begin{array}{l}
(a) \ \ H^0(\calO_{\tX'}(-\pi^*x-r_h' -E_{new}) \subsetneq  H^0(\calO_{\tX'}(-\pi^*x-r_h'), \ \mbox{and} \\
(b) \ \ H^0(\tX', \Omega^2_{\tX'}(\pi^* x +r_h'+E_I)) \subsetneq
H^0(\tX', \Omega^2_{\tX'}(\pi^* x +r_h'+E_I+E_{new})).
\end{array}\right.$$
Part {\it (a)}  means the following: there exists a global section $s_1\in H^0(\tX', \calO_{\tX'}(-r_h'))$
such that ${\rm div}_{E'}(s_1)\geq \pi^*x$, and in this inequality the $E_{new}$--coordinate entries are equal.
By part {\it (b)}, there exists a global section  $s_2\in H^0(\tX', \Omega^2_{\tX'}(r'_h)) $ such that
${\rm div}_{E'}(s_2)\geq -\pi^*x-E_I-E_{new}$ and the $E_{new}$--coordinate entries are equal.

Therefore, the global section  $s_1s_2\in H^0(\tX', \Omega^2_{\tX'})$ has the property that
${\rm div}_{E'}(s_1s_2)\geq -E_I-E_{new}$ with equality at the $E_{new}$ coordinate.
In particular,  by duality (\ref{eq:duality}) we obtain that  in $\tX'$ the following
strict inequality  holds:
\begin{equation}\label{eq:ROSSZ}
h^1(\calO_{E_I+E_{new}})>h^1(\calO_{E_I}) \ \ (\cV'=\cV\cup\{new\}, \ I\subset \cV).\end{equation}
But if the link is a rational homology sphere then both left and right hand sides are zero,
i.e.  this strict inequality cannot happen.

\bekezdes In particular, for any $I\subset \cV$ either $\{v_0,v_1\}\not\subset  I$, or in the opposite case
either\, {\bf Case 1} or {\bf Case 2} applies.  Hence,
in any case,
$\pi_{\R}^{-1}(x,I)^\circ \cap S_n(\phi')$ is contractible. Therefore, $S_n(\phi)$ and $S_n(\phi')$
have the same homotopy type by the argument from the end of \ref{bek:proof1}.

\vspace{2mm}

\underline{{\bf Case B.II.}} Assume that $r_h'=\pi^*r_h-E_{new}$.

In turns out that this case is very similar to the case B.I: compared with that case
all the $E_{new}$--coefficients should be shifted by one. However, we have to go through all the
verifications  step by step.

Firstly, for $a\leq 1$,
$$\hh'(\pi^*x+aE_{new})=
\dim \frac{H^0(\calO_{\tX'}(-\pi^*r_h+E_{new}))}{
H^0(\calO_{\tX'}(-\pi^*x-\pi^*r_h-aE_{new}+E_{new}))}.$$
Since
$H^0(\calO_{\tX'}(-\pi^*r_h+E_{new}))=H^0(\calO_{\tX'}(-\pi^*r_h))$,
and for $a\leq 1$ (by \ref{bek:1})
$$H^0(\calO_{\tX'}(-\pi^*x-\pi^*r_h-aE_{new}+E_{new}))=
H^0(\calO_{\tX'}(-\pi^*x-\pi^*r_h))$$
we get
\begin{equation}\label{eq:MON1B}
 \hh'(\pi^*x+aE_{new}) \ \left\{ \begin{array}{l}
 = \hh(x) \ \mbox{ for any $a\leq 1$} \\
 \mbox{is increasing for $a\geq 1$}.\end{array}\right.
\end{equation}
Next, for $a\geq 0$,  in the cohomology  exact sequence of
 $$0\to \calO_{aE_{new}}(-\pi^*x-\pi^*r_h+E_{new})\to \calO_{\pi^*x+aE_{new}}(-r_h')\to \calO_{\pi^*x}(-r_h')\to 0$$
 one has $h^1(\calO_{aE_{new}}(-\pi^*x-\pi^*r_h+E_{new}))=0$. Indeed, since ${\rm Pic}^0(aE_{new})=0$,
 $h^1(\calO_{aE_{new}}(-\pi^*x-\pi^*r_h+E_{new}))=h^1(\calO_{aE_{new}}(E_{new}))$, whose vanishing follows by induction on $a$.
 Therefore, for $a\geq 0$, \begin{equation}\label{eq:B22}
 h^1(\calO_{\pi^*x+aE_{new}}(-r_h'))=h^1( \calO_{\pi^*x}(-r_h')).\end{equation}
 On the other hand, from the exact sequence
 $$0\to \calO_{\pi^*x}(-\pi^*r_h)\to \calO_{\pi^*x+E_{new}}(-\pi^*r_h+E_{new})\to \calO_{E_{new}}(-\pi^*r_h+E_{new})\to 0$$
 we obtain $h^1(\calO_{\pi^*x}(-\pi^*r_h))=h^1(\calO_{\pi^*x+E_{new}}(-\pi^*r_h+E_{new}))$, which equals
$h^1(\calO_{\pi^*x}(-r_h'))$ by (\ref{eq:B22}). Hence
  \begin{equation}\label{eq:MON2B}
 h^1(\calO_{\pi^*x+aE_{new}}(-r_h')) \ \left\{ \begin{array}{l}
  \mbox{is increasing  for $a\leq  0$}, \\
= h^1(\calO_x(-r_h)) \ \mbox{ for any $a\geq  0$}. \\
\end{array}\right.
\end{equation}
These combined provide
\begin{equation}\label{eq:HOM2aB}
a\mapsto w_0'(\pi^*x+aE_{new}) \ \left\{ \begin{array}{l}
\mbox{is  decreasing for $a\leq  0$},\\
 = w_0(x) \ \mbox{ for  $a= 0$ and $a=1$,} \\
 \mbox{is increasing for $a\geq 1$}.\end{array}\right.
\end{equation}
 Here it is convenient is to take $c=\infty$, hence we compare the two infinite rectangles (first quadrants).

 Again, if $w_0'(\pi^*x+aE_{new})\leq n$, then
$w_0(x)\leq n$ too. Hence  the projection $\pi_{\R}$ in the direction of $E_{new}$
induces a  map $\pi_{\R}:S_n(\phi')\to S_n(\phi)$.
We need to prove that this is a   homotopy equivalence  with all fibers non-empty and contractible.

First we verify that  $\pi_{\R}:S_n(\phi')\to S_n(\phi)$ is onto.

If $x\in S_n(\phi)$ then $w_0(x)\leq n$, hence  by (\ref{eq:HOM2aB}) $w_0'(\pi^*x)=w_0(x)\leq n$ too, hence
$x\in {\rm im}(\pi_{\R})$.

If $(x,I)\subset S_n(\phi)$ $(I\subset \calv)$ then $w_0(x+E_{I'})\leq n$ for any $I'\subset I$. For such $I'$
 we have the identity (\ref{eq:eps3}) with $\epsilon =|I'\cap \{v_0,v_1\}|\subset \{0,1,2\}$.

 Assume that $\{v_0,v_1\}\subsetneq I$. Then we claim that $(\pi^*x+E_{new}, I)$ is in $S_n(\phi')$ and it projects isomorphically onto $(x,I)$.  Indeed, in this case $\epsilon\in\{0,1\}$ and by (\ref{eq:HOM2aB})
 $$w_0'(\pi^*x+E_{new}+E_{I'})=w_0'(\pi^*(x+E_{I'})-\epsilon E_{new}+E_{new})=w_0(x+E_{I'})\leq n.$$

 Hence in the sequel we assume that $\{v_0,v_1\}\subset I$.  Then we proceed as in \ref{bek:3}. Again, there are two
 cube--candidates to lift $(x,I)$.

 One of them is $(\pi^*x+E_{new},I)$. However, though $\pi^*x+E_I+2E_{new}$ and  $\pi^*x+E_I+3E_{new}$ are in $S_n(\phi')$ but
 the vertex   $\pi^*x+E_I+E_{new}$ of  $(\pi^*x+E_{new},I)$  might not be part of $S_n(\phi')$.

   The second candidate is $(\pi^*x+2E_{new},I)$, but this case is also obstructed:
   $\pi^*x$ and $\pi^*x+E_{new}$ are in $S_n(\phi')$ but the vertex $\pi^*x+2E_{new}$ of
    $(\pi^*x+2E_{new},I)$ not necessarily.

Hence, again we have to analyse three case, the analogues of \ref{bek:A}, \ref{bek:B} and \ref{bek:C}.

\noindent {\bf Case 1.} We assume that $w_0'(\pi^*x+E_{new})= w_0'(\pi^*x+2E_{new})$. Then similarly as in  \ref{bek:A}
one can show that  $(\pi^*x+2E_{new},I)\subset S_n(\phi')$.

\noindent {\bf Case 2.}  We assume that $w_0'(\pi^*x+E_I+ E_{new})= w_0'(\pi^*x+E_I+2E_{new})$. Then similarly as in  \ref{bek:B}
one can show that  $(\pi^*x+E_{new},I)\subset S_n(\phi')$.

 \noindent {\bf Case 3.} Finally we show that either Case  1 or Case 2 must hold. Indeed, if
 not, that is, if
 $$\left\{ \begin{array}{l}
\hh'(\pi^*x+E_{new})<\hh'(\pi^*x+2E_{new}), \ \mbox{and} \\
h^1(\calO_{\pi^*x+E_I+E_{new}}(-r_h'))< h^1(\calO_{\pi^*x +E_I+2E_{new}}(-r_h')),\end{array}\right.$$
 then we get a contradiction similarly as in \ref{bek:C}.

 \vspace{2mm}

 Having the surjectivity  $\pi_{\R}:S_n(\phi')\to S_n(\phi)$,
 the  homotopy equivalence  is proved as in the previous cases.
\end{proof}

\begin{definition}
In the sequel we will use for $\bH^*_{an,h}(\phi)$ the notation  $\bH^*_{an,h}(X,o)$ as well.
It is called the {\it analytic lattice cohomology of $(X,o)$} associated with $h\in H$.
We also set $\bH^*_{an}(X,o):=\oplus _{h\in H}\bH^*_{an,h}(X,o)$. It is called the
{\it equivariant analytic lattice cohomology of $(X,o)$}.

We adopt the notation $\RR_{an,h}(X,o)$ for the graded root as well.%
\end{definition}


\begin{remark}
In order to run the equivariant version (indexed by $H$) we need the existence of the
universal abelian covering, hence we need the finiteness of $H_1(M,\Z)$, i.e.
we need to require that  the link is a rational homology sphere. On the other hand,  if we wish to
study only the analytic lattice cohomology associated with $h=0$ (that is, with $\calO_{X,o}$), then we do not need
to consider the universal abelian covering. In that  case, as  the above proof shows, in order to prove the stability
of $\bH^*_{an,h=0}(\phi)$ it is enough to assume that $\Gamma$ is a tree (this is enough to  conclude that (\ref{eq:ROSSZ}) cannot happen).
For details for the non-equivariant case see \cite{AgNe1}.
\end{remark}

\subsection{The `Combinatorial Duality Property' of the pair $(\hh, \hh^\circ)$}\label{ss:anCDP}

The following property is needed  in the Euler characteristic computation.

\begin{lemma}\label{lem:hsimult} Assume that the link is a rational homolog sphere.
Then there  exists no  $l\in L_{\geq 0}$ and $v\in\calv$  such that  the differences
$\hh(l+E_v)-\hh(l)$ and $\hh^\circ (l)-\hh^\circ (l+E_v)$ are simultaneously strict positive.
\end{lemma}
\begin{proof}
%
If $\hh(l+E_v)>\hh(l)$ then the inclusion $H^0(\calO_{\tX}(-l-r_h-E_v))\subset H^0(\calO_{\tX}(-l-r_h))$
is strict. This means that there exists
a section $s_1\in H^0(\calO_{\tX}(-r_h))$ with ${\rm div}_E(s_1)\geq l$, where the
$E_v$-coordinate is  $({\rm div}_E(s_1))_v = l_v$.

Similarly, if $\hh^\circ (l)>\hh^\circ(l+E_v)$ then
the inclusion $H^0(\Omega^2_{\tX}(l+r_h)\subset H^0(\Omega^2_{\tX}(l+r_h+E_v))$ is strict, that is,
 there exists
a section $s_2\in H^0(\Omega^2_{\tX}(r_h))$ with ${\rm div}_E(s_2)\geq -l-E_v$ and  the
$E_v$-coordinate is  $({\rm div}_E(s_2))_v = -l_v-1$.

Therefore, the section $s_1s_2\in H^0(\Omega^2_{\tX})$ satisfies
 ${\rm div}_E\ (s_1s_2) \geq -E_v$  and
  $({\rm div}_E (s_1s_2))_v = -1$. This implies $H^0(\Omega_{\tX}^2(E_v))/H^0(\Omega_{\tX}^2)\not=0$, or,
  by (\ref{bek:LauferDual}), $h^1(\calO_{E_v})\not=0$. This last fact contradicts  $H^1(M,\Q)=0$.
\end{proof}

\subsection{The Euler characteristic $eu(\bH^*_{an,h}(X,o))$}\label{ss:anEu}

\bekezdes Lemma \ref{lem:hsimult} will allow us to determine the  Euler characteristic $eu(\bH^*_{an,h}(X,o))$
of the analytic lattice cohomology by a combinatorial argument. Surprisingly, this Euler characteristic automatically
equals the Euler characteristic of  path cohomolgies  associated with any increasing path
(this equality definitely does not hold
in  the topological versions of the corresponding lattice cohomologies).

First, let us fix the notations.
In the sequel
we will also consider
for any increasing path $\gamma$ connecting 0 and $c$
 (that is, $\gamma=\{x_i\}_{i=0}^t$, $x_{i+1}=x_i+E_{v(i)}$, $x_0=0$ and $x_t=c$, $c\geq Z_{coh,h}$)
 the  path lattice cohomology $\bH^0(\gamma,w(h))$ as in \ref{bek:pathlatticecoh}, associated with the weight function
 (depending on $h\in H$).
Accordingly,  we have the numerical
Euler characteristic
 $eu(\bH^0(\gamma,w(h)))$ as well.

Then  Theorem \ref{th:comblattice} implies the following.

\begin{theorem}\label{th:euANLAT} Assume that the link is a $\Q HS^3$. Then
$eu(\bH^*_{an,h}(X,o))=p_{g,h}(X,o)$
 for any $h\in H$. Furthermore, for any increasing path $\gamma$ connecting 0 and $c$ (where
$c\geq Z_{coh,c}$)
 we also have $eu(\bH^*_{an}(\gamma,w(h)))=p_{g,h}$.
\end{theorem}
\begin{proof}
We claim that the assumptions of Theorem \ref{th:comblattice} are satisfied. Indeed,
the CDP was verified in \ref{lem:hsimult}, while the stability property of $\hh$  follows
since it is associated with a filtration.
\end{proof}
This  in particular means that $\bH^*_{an,h}(X,o)$ is a {\it categorification  of the equivariant geometric genus},
that is, it is a graded cohomology module whose Euler characteristic is $p_{g,h}$.

\subsection{Weighted cubes and the Poincar\'e series $P(\bt)$.} \label{ss:Poinc}

\ Assume that $c=\infty$, i.e.
 $R(0,c)=L_{\geq 0}$. Let us denote the weight function associated with $h\in H$ by
 $w_{an, h}$, in order to emphasise the $h$--dependence.

The reader is invited to review the definition of the Poincar\'e series $P(\bt)$ from (\ref{eq:4}).
That identity together with part  {\it (b)} of Theorem \ref{th:comblattice}
 show that   the analytic Poincar\'e series
associated with the divisorial filtration of the local ring $\cO_{X_a,o}$
has the following interpretation
in terms of  the (analytic) weighted cubes:
$$P(\bt)=\sum_{h\in H}\ \sum_{l\geq 0}\ \sum _{I\subset \calv}\  (-1)^{|I|+1} w_{an,h}((l,I))\, \bt^{l+r_h}.$$
The above formula can be compared with its topological analogue.
One defines a topological zeta (Poincar\'e) series $Z(\bt)$ from $\Gamma$,
and there is an identical formula for $Z(\bt)$, where $w_{an}$ is replaced by
$w_{top}$,  cf. \cite{NJEMS}.

\begin{question} Assume that the link of the  universal abelian covering $(X_a,0)$ of $(X,o)$ is a rational homology sphere.
Then the analytic lattice cohomology $\bH^*_{an,0}(X_a,o)$ of $(X_a,o)$ associated with the trivial element of $H_1(M(X_a),\Z)$ is
well--defined, and it is the categorification of $p_g(X_a,o)$.

Furthermore, for every $h\in H$ we have the analytic lattice cohomology $\bH^*_{an,h}(X,o)$ of $(X,o)$.
$\bH^*_{an,h}(X,o)$ is the categorification of $p_{g,h}(X.o)$.
Recall also that $p_g(X_a,o)=\sum_{h\in H} p_{g,h}(X,o)$.

Is there a  relationship between
$\bH^*_{an,0}(X_a,o)$  and the collection $\{\bH^*_{an,h}(X,o)\}_{h\in H}$ ?
\end{question}

\section{Comparison of $\bH^*_{an}(X,o)$ with $\bH^*_{top}(M)$}\label{an:Comp}

\subsection{}  Above, for every $h\in H$, we defined the analytic lattice cohomology $\bH^*_{an,h}(X,o)$
associated with $L_{\geq 0}$ and the weight function $w_{an, h}:L_{\geq 0}\to \Z$.

Similarly, for any $h\in H$ we can consider the characteristic element $k=-Z_K+2r_h$  and
the topological lattice cohomology associated with $k=-Z_K+2r_h$ and $L$
via the weight function $l\mapsto -(l, l-Z_K+r_h)/2$. Let us denote it by
$\bH^*_{top,h}(M)$.  On the other hand, in \ref{bek:rhred} we proved that
$\bH^*_{top,h}(M)\simeq \bH^*_{top,h}(M, L_{\geq 0})$, where the second cohomology is associated with
the same $k$ but with lattice points only on $L_{\geq 0}$.
The advantage of  $\bH^*_{top,h}(M, L_{\geq 0})$ is that it is defined on the same set of lattice points as the
analytic $\bH^*_{an,h}(X,o)$.

Let us compare these two objects.
First,  we compare the analytic and topological
weight functions (both defined on $L_{\geq 0}$).
 Consider  the exact sequence $$0\to \calO_{\tX}(-l-r_h)\to \calO_{\tX}(-r_h)\to \calO_l(-r_h)\to 0,$$
and in its cohomology long exact sequence the morphism  $\alpha_h(l):H^0(\calO_{\tX}(-r_h))\to H^0(\calO_l(-r_h))$.
Then, $\hh(l)+{\rm coker}(\alpha_h(l))=h^0(\calO_l(-r_h))$, or,
$w_{an,h}=\chi(\calO_l(-r_h))-{\rm coker}( \alpha_h(l))$.

But $\chi(\calO_l(-r_h))=\chi(l)-(l,r_h)=\chi_{-Z_K+2r_h}(l)=w_{top,h}(l).$ Hence
$$w_{an,h}(l)=w_{top,h}(l)-{\rm coker}( \alpha_h(l)) \ \ \mbox{for any $l\in L_{\geq 0}$}.$$
\begin{corollary}
If $\alpha_h(l)$ is surjective for every $l\in L_{\geq 0}$  then $\bH^*_{an,h}(X,o)$ and $\bH^*_{top,h}(M)$ are
isomorphic as graded $\Z[U]$--modules. In particular, in such a case their Euler characteristics also coincide:
$$p_{g,h}=\sw_{\sigma[k]}(M)-(k^2+|\cV|)/8, \ \ \mbox{where $k=-Z_K+2r_h$}.$$
\end{corollary}

In general, $w_{an,h}\leq w_{top,h}$. Recall that
 $S_{an,h,n}=\cup\{\square \,:\, w_{an,h}(\square)\leq n\}$ and
$S_{top,h,n}=\cup\{\square \,:\, w_{top,h}(\square)\leq n\}$.  Therefore
$S_{top,h,n}\subset S_{an,h,n}$ for any $n\in\Z$.
In particular, we have  a graded $\Z[U]$--module morphism
$$\mathfrak{H}^*_h:\bH^*  _{an,h}(X,o)\to
\bH^*_{top,h}(M)$$
and a morphism of graded roots
$$\mathfrak{r}^*_h:\RR_{top,h}(X,o)\to
\RR_{an,h}(M).$$

\begin{problem} (a) For a fixed topological type find all the possible graded $\Z[U]$--modules $\{\bH^*_{an}\}_{an,h}$,
associated with all the possible analytic structures supported on that topological type.

(b) For a fixed topological type (hence for a fixed $\bH^*_{top,h}(M)$) and analytic type $(X,o)$ supported on it
find special properties of
 $\bH^*_{an,h}(X,o)$ (and of the morphism $\bH^*_{an,h}(X,o)\to \bH^*_{top,h}(M)$), which might characterize the
 classification  from part (a).

\end{problem}

\section{Preparation for the reduction theorem. The topological reduction.}

\subsection{What is the aim
 of a Reduction Theorem?}
The definition of a lattice cohomology $\bH^*(T,w)$ is based on the choice of the following objects:
a lattice $L=\Z^s$, a convenient union of cubes $T\subset \R^s$, a weight function $w:T\cap \Z^s\to Z$.
In general, $s$, the rank of $L$, can be large, and the direct computations are very hard.
By Reduction Theorem we replace these starting objects by a new collection $(\bar{L}, \bar{T}, \bar{w})$ such that
 ${\rm rank}(\bar{L})<{\rm rank}(L)$ and $\bH^*(T,w)=\bH^*(\bar{T},\bar{w})$.

The Reduction Theorems associated with the topological  lattice cohomology are based on the following observation: the reduced
cohomologies are vanishing if and only if  $M$ is the link of a rational singularity.
Rationality can be characterized by properties of graphs (see below).
In general, we wish to `eliminate' parts/subgraphs,  which behave like rational graphs.
Technically, the procedure runs as follows: we choose $s'$ vertices (the bad vertices) such that by the modification of their Euler numbers
we get a rational graph. Then there is a reduction to rank $s'$.

\bekezdes \label{ex:LatHomRat} {\bf Rational graphs.}
Recall that $(X,o)$ is called rational if $p_g=0$. By a result of Artin \cite{Artin62,Artin66}
$p_g=0$ if and only if
 $\chi(l)\geq 1$ for all $ l\in L_{>0}$  (hence it is a topological property of $M$  readable from $\Gamma$).
The links of
 any  rational singularity is a
 rational homology sphere. 
 The class of  rational  graphs is closed while taking subgraphs or/and
decreasing the Euler numbers $E_v^2$.

%
%

\subsection{Measure of non-rationality. `Bad' vertices}\label{ss:BadVer} \cite{NOSz,LN1,AgNe1,Book}\

Recall
that decreasing all the Euler numbers of  a tree, with all the  vertices decorated by $g_v=0$,
we  obtain a rational graph.
The next definition aims to
identify those vertices where such a decrease is really necessary.

\begin{definition}\label{def:SWrat} Let $\Gamma$ be a resolution graph such that $M$ is a rational homology sphere.

A subset of vertices
$\ocalj=\{v_1,\ldots, v_{\overline{s}}\}\subset \cV$ is called {\it B--set},
if by replacing the Euler numbers
 $e_v=E_v^2$ indexed by $v\in \ocalj$ by some more negative integers
$e'_v\leq e_v$ we get a rational graph.

A graph is called AR-graph (`almost rational graph') if it admits a B--set of cardinality $\leq 1$.
\end{definition}
\begin{example}\label{ex:*sets}
(a) A possible $B$--set can be chosen in many different ways, usually
it is not determined uniquely even if it is minimal with this property.
Usually we allow non-minimal $B$--sets as well.

(b) If $H_1(M,\Q)=0$ then the set of nodes is a B--set.
Hence  any  star-shaped graph (with $H_1(M,\Q)=0$) is AR.
Other  AR families are: rational and elliptic graphs and graphs of superisolated singularities associated with a
rational unicuspidal curve \cite{NOSz,NGr}.

(c)
The class of  AR graphs is closed while taking subgraphs or/and
decreasing the Euler numbers.

\end{example}

\bekezdes\label{bek:XI}
 {\bf The definition of the lattice points $x(\bar{l})$.}
Assume that $\ocalj:=\{v_k\}_{k=1}^{\overline{s}}$ is a subset of $\calj$.
Then we split the set of vertices $\calj$ into the disjoint union $\overline{\calj}\sqcup\calj^*$.
Let $\{m_v(x)\}_v$ denote  the coefficients of a cycle $x\in L\otimes \setQ$, that is $x=\sum_{v\in\calj}m_v(x)E_v$.
We also fix
 $h\in H$ and  the  representative $s_h\in L'$.

 Our goal is to define some universal cycles  $x(\bar{l})\in L$
 associated with $\bar{l}\in L(\ocalj)$ and $h\in H$.

\begin{proposition}\label{lemF1} \ \cite[Lemma 7.6]{NOSz}, \cite{LN1}
For any
$\bar{l}:=\sum_{v\in \ocalj}\ell_v
E_v\in L(\ocalj)$
there exists a unique cycle
$x(\bar{l})\in L$ (depending also on $h$)
satisfying the next properties:
\begin{itemize}
\item[(a)] \ \ $m_{v}(x(\bar{l}))=\ell_v$ \ for any distinguished  vertex $v\in\ocalj$;
\vspace{1mm}
\item[(b)] \ \ $(x(\bar{l})+s_h,E_v)\leq0$ \ for every `non-distinguished vertex' $v\in\calj^*$;
\vspace{1mm}
\item[(c)] \ \ $x(\bar{l})$ is minimal with the two previous properties (with respect to $\leq$).
\end{itemize}
\end{proposition}

\bekezdes
Note  that the definition of an B--set does not involve
any $k\in {\rm Char}$, hence such a set can be uniformly used for any
$k_r$. In this section we  fix such an B--set $\ocalj\subset \calv$
as in \ref{def:SWrat} (with cardinality $\bar{s}$)  and any $k_r\in {\rm Char}$.
Then,   for each $\bar{l}=\sum_{v\in \ocalj} \ell_vE_v\in L(\ocalj)$,
with every  $\ell_v\geq 0$, we define the  universal  cycle $x(\bar{l})$ associated with
$\bar{l}$ and $s_h$ (where $k_r=-Z_K+2s_h$) as in \ref{lemF1}.

Our goal is to replace the cubes of the lattice $\R^s$ (or from $(\R_{\geq 0})^s$)
with cubes from $(\R_{\geq 0})^{\bar{s}}$. In particular,  we need to define the new weights.
Define   the function
$\overline{w}_0:(\Z_{\geq 0})^{\bar{s}}\to \Z$ by
$\overline{w}_0(\bar{l}):=\chi_{k_r}(x(\bar{l}))$.
 Then $\overline{w}_0$ defines a set $\{\overline{w}_q\}_{q=0}^{\bar{s}}$ of compatible weight functions  as in \ref{9dEF1},
$\overline{w}_q(\square)=
 \max\{\overline{w}_0(v)\,:\, v \ \mbox{ is a vertex of $\square$}\}$.
 This system   is   denoted by $\overline{w}[k_r]$.
Let us denote the associated  lattice cohomology by
$H^*((\R_{\geq 0})^{\bar{s}},\overline{w}[k_r])$.

\begin{theorem}\label{th:red} {\bf (Topological Reduction Theorem)}  \cite{LN1} Assume that $\ocalj$ is an B--set. Then
 there exists  a graded $\Z[U]$-module isomorphism
\begin{equation}\label{eq:reda}
\bH^*((\R_{\geq 0})^s,k_r)\cong\bH^*((\R_{\geq 0})^{\bar{s}},\overline{w}[k_r]).
\end{equation}
\end{theorem}

\section{Analytic Reduction Theorem}

\subsection{Analytic reduction theorem}\label{ss:anRT}

\bekezdes
Our next goal is to prove a `Reduction Theorem', the analogue of the topological Theorem
\ref{th:red}.
Via such a result, the rectangle $R=R(0,c)$ can be replaced  by another rectangle sitting in a
lattice of smaller rank. The procedure starts with identification of a set of `bad' vertices,
see \ref{ss:BadVer}.
In the topological context the possible choice of $\overline{\calv}$ was dictated by combinatorial
properties of $\chi$ with a special focus on the topological characterization of rational germs.
In the present context we start with certain  analytic properties of 2-forms (which reflects the dominance
of $\overline{\calv}$ over $\calv^*$).
(Note that $p_g=0$ if and only  if $H^0(\tX\setminus E, \Omega^2_{\tX})=H^0(\tX, \Omega^2_{\tX})$.)

In this section we assume that the link is a rational homology sphere.
\begin{definition}\label{def:DOMAN}
We say that  $\overline{\calv}$ is an B$_{an}$--set
if it satisfy the following
property: if 
some differential form
$\omega\in H^0(\tX\setminus E, \Omega_{\tX}^2) $ satisfies
$({\rm div}_E\omega) |_{\overline{\calv}}\geq -E_{\overline{\calv}}$ \
then necessarily
$\omega\in H^0(\tX, \Omega_{\tX}^2)$.
By (\ref{bek:LauferDual}) this is equivalent with the vanishing  $h^1(\calO_Z)=0$ for any
$Z=E_{\overline{\calv}}+l^*$, where $l^*\geq 0$ and it is supported on $\calv^*$.
\end{definition}


\begin{lemma}\label{lem:AnNodes} \cite{AgNe1}
Any $B$--set is a B$_{an}$--set.  
 \end{lemma}

 \begin{example}\label{ex:Ran} By the above lemma,
 the set $\overline{\calv}={\mathcal N}$ of nodes is an B$_{an}$--set.
 Moreover, if $\{\overline{v}\}$ is the B--set of an AR graph, then it is an B$_{an}$--set as well.
 \end{example}

\bekezdes Associated with a disjoint  decomposition  $\calv=\overline{\calv}\sqcup \calv^*$,  we write
any  $l\in L$
as $\overline{l}+l^*$, or $(\overline{l}, l^*)$,
where $\overline {l}$ and  $l^*$ are  supported on $\overline{\calv}$ and
 $\calv^*$ respectively.
 Fix any $c\geq Z_{coh,h}$ and set $R=R(0,c)$ as above.
 We also write  $\overline{R}$ for the rectangle $R(0, \overline{c})$, the $\overline{\calv}$-projection of $R$.
For any $\overline {l}\in \overline {R}$ define  the weight function
$$\overline{w}_0(\overline{l})=\hh(\overline{l})+\hh^\circ (\overline{l}+c^*)-p_{g,h}
=\hh(\overline{l})-h^1(\calO_{\overline{l}+c^*}(-r_h)).$$
Consider all the cubes of $\overline{R}$ and the weight function
$\overline{w}_q:\calQ_q(\overline{R})\to \Z$ defined  by $ \overline{w}_q(\square_q)=\max\{w_0(\overline{l})\,:\, \overline{l} \
 \mbox{\,is any vertex of $\square_q$}\}$.

\begin{theorem}\label{th:REDAN} {\bf Reduction theorem for the analytic lattice cohomology.}
If  $\overline{\calv}$ is an $B_{an}$--set
then there exists a graded $\Z[U]$--module isomorphism
$$\bH^*_{an}(R,w)\simeq \bH^*_{an}(\overline{R}, \overline{w}).$$
\end{theorem}
\begin{proof} For any $\cali\subset \calv$ write $c_\cali$ for the $\cali$-projection of $c$.

We proceed by induction, the proof will be given in $|\calv^*|$ steps.
For any $\overline{\calv}\subset \cali\subset \calv$ we create the inductive setup.
We write $\cali^*=\calv\setminus \cali$, and according to the disjoint union
$\cali\sqcup \cali^*=\calv$ we  consider the coordinate decomposition
$l=(l_\cali,l_{\cali^*})$. We also set $ R_\cali=R(0, c_\cali)$ and the weight function
$$w_\cali(l_\cali)=\hh(l_\cali)+\hh^\circ(l_\cali+c_{\cali^*})-p_{g,h}. $$
Then  for $\overline{\calv}\subset \cali\subset \cJ\subset \calv$, $\cJ=\cali\cup \{v_0\}$
($v_0\not\in\cali$),
we wish to prove that $\bH^*(R_\cali, w_\cali)=\bH^*(R_{\cJ}, w_{\cJ})$.
For this consider the projection $\pi_{\R}:R_{\cJ}\to R_{\cali}$.

For any fixed $y\in R_\cali$ consider the fiber $\{y+tE_{v_0}\}_{0\leq t\leq c_{v_0},\ t\in \Z}$.

Note that $t\mapsto \hh(y+tE_{v_0})$  is increasing. Let $t_0=t_0(y)$
be the smallest value $t$ for which
$\hh(y+tE_{v_0})< \hh(y+(t+1)E_{v_0})$.
If $t\mapsto \hh(y+tE_{v_0})$ is constant then we take $t_0=c_{v_0}$.
If $t_0<c_{v_0}$, then $t_0$ is characterized by the existence of a global section
\begin{equation}\label{eq:1RED}
s_1\in H^0(\calO_{\tX}(-r_h)) \ \ \mbox{with} \ \
({\rm div}_Es_1)|_\cali\geq y, \ \ \  ({\rm div}_Es_1)_{v_0}=t_0.
\end{equation}
Symmetrically,  $t\mapsto \hh^{\circ} (y+c_{\cJ^*}+ tE_{v_0})$  is decreasing. Let $t_0^\circ=t_0^\circ (y)$
be the smallest value $t$ for which
$\hh^{\circ} (y+c_{\cJ^*}+tE_{v_0})=\hh^{\circ} (y+c_{\cJ^*}+(t+1)E_{v_0})$. The value
$t_0^\circ$  is characterized by the existence of a section
\begin{equation}\label{eq:2RED}
s_2 \in H^0(\tX\setminus E, \Omega_{\tX}^2(r_h)) \ \ \mbox{with} \ \
({\rm div}_E s_2) |_\cali\geq - y, \ \ \  ({\rm div}_Es_2 )_{v_0}=-t^{\circ}_0.
\end{equation}
This shows that there exist a form $\omega=s_1s_2\in H^0(\tX\setminus E, \Omega_{\tX}^2)$ such that
$({\rm div}_E\omega )|_\cali\geq 0$ and $({\rm div}_E\omega )_{v_0}=t_0-t^{\circ}_0$.
By the B$_{an}$ property we necessarily must have $t_0-t^{\circ}_0\geq 0$.
Therefore, the weight $t\mapsto w_{\cJ}(y+tE_{v_0})=\hh(y+tE_{v_0})+\hh^{\circ } (y+tE_{v_0}+c_{\cJ^*})-p_{g,h}$
is decreasing for $t\leq t_0^\circ$, is increasing for $t\geq t_0$. Moreover, for $t_0^\circ \leq t\leq t_0$
it  takes the
constant value $\hh(y)+\hh^{\circ } (y+c_{v_0}E_{v_0}+c_{\cJ^*})-p_{g,h}=w_{\cali}(y)$.

Next we fix $y\in R_\cali$ and some $I\subset \cali$ (hence a cube $(y,I)$ in $R_{\cali}$).
We wish to compare the intervals
$[t_0^\circ (y+E_{I'}), t_0 (y+E_{I'})]$ for all subsets $I'\subset I$. We claim that they have at least one
common element (in fact, it turns out that $t_0(y)$ works).

Note that $\hh(y+tE_{v_0})=\hh(y+(t+1)E_{v_0})$ implies $\hh(y+tE_{v_0}+E_{I'})=\hh(y+(t+1)E_{v_0}+E_{I'})$
for any $I'$,
hence $t_0(y)\leq t_0(y+E_{I'})$. In particular, we need to prove that $t_0(y)\geq t_0^\circ (y+E_{I'})$.
Similarly as above, the value $t_0^\circ(y+E_{I'})$  is characterized by the existence of a form
\begin{equation*}
s_{I'} \in H^0(\tX\setminus E, \Omega_{\tX}^2(r_h)) \ \ \mbox{with} \ \
({\rm div}_E s_{I'}) |_\cali\geq - y-E_{I'}, \ \ \  ({\rm div}_E s_{I'} )_{v_0}=-t^{\circ}_0(y+E_{I'}).
\end{equation*}
Hence the  from $\omega_{I'}=s_1s_{I'}\in H^0(\tX\setminus E, \Omega_{\tX}^2)$ satisfies
${\rm div}_E\omega_{I'} |_\cali\geq -E_{I'}$ and $({\rm div}_E\omega_{I'} )_{v_0}=t_0(y)-t^{\circ}_0(y+E_{I'})$.
By the B$_{an}$ property we must have $t_0(y)-t^{\circ}_0(y+E_{I'})\geq 0$.

Set $S_{\cJ,n}$ and $S_{\cali,n}$ for the lattice spaces defined by $w_\cJ$ and $w_\cali$. If
$y+tE_{v_0}\in S_{\cJ,n}$ then $w_\cJ(y+tE_{v_0})\leq n$, hence  by the above discussion $w_\cali(y)\leq n$ too.
In particular, the projection $\pi_{\R}:R_\cJ\to R_\cali$ induces a map $S_{\cJ,n}\to S_{\cali,n}$.
We claim  that it is a homotopy equivalence. The argument is similar to the proof from
\ref{th:annlattinda} via the above preparations.
\end{proof}

\end{document}